\documentclass[12pt,reqno]{amsart}
\usepackage{amsmath,amstext,amssymb,amscd,hyperref}
\usepackage{verbatim}
\usepackage{enumerate}
\usepackage{mathrsfs}
\usepackage[dvipsnames,usenames]{xcolor}

\usepackage{amsthm}

\newtheorem{theorem}{Theorem}[section]

\newtheorem{corollary}[theorem]{Corollary}
\newtheorem{proposition}[theorem]{Proposition}
\newtheorem{assumption}[theorem]{Assumption}

\theoremstyle{definition}
\newtheorem{definition}[theorem]{Definition}

\theoremstyle{remark}
\newtheorem{remark}[theorem]{Remark}

\numberwithin{equation}{section}

\def\a{\alpha}

\def\l{\lambda}

\def\t{\tau}

\def\G{\Gamma}
\def\O{\Omega}
\def\e{\epsilon}
\def\p{\phi}

\newcommand{\reals}{\mathbb{R}}
\newcommand{\naturals}{\mathbb{N}}

\newcommand{\innerprod}[1]{\left\langle#1\right\rangle}
\newcommand{\norm}[1]{\left\|#1\right\|}
\newcommand{\abs}[1]{\left|#1\right|}
\newcommand{\ds}{\displaystyle}

\author[Y. Guo]{Yanqiu Guo}
\address{Department of Computer Science and Applied Mathematics \\ Weizmann Institute of Science\\
Rehovot 76100, Israel} \email{yanqiu.guo@weizmann.ac.il}

\author[M. A. Rammaha]{Mohammad A. Rammaha}
\address{Department of
Mathematics \\ University of Nebraska-Lincoln \\
Lincoln, NE  68588-0130, USA} \email{rammaha@math.unl.edu}

\author[S. Sakuntasathien]{Sawanya Sakuntasathien}
\address{ Department of Mathematics \\ Faculty of Science \\ Silpakorn University  \\
Nakhonpathom, 73000, Thailand} \email{sawanya.s@su.ac.th}

\author[E. S. Titi]{Edriss S. Titi}
\address{Department of Mathematics and Department of Mechanical and Aerospace Engineering\\
University of California, Irvine, California 92697-3875, USA and Department of Computer
Science and Applied Mathematics, Weizmann Institute of Science, Rehovot 76100, Israel} \email{etiti@math.uci.edu and edriss.titi@weizmann.ac.il}

\author[D. Toundykov]{Daniel Toundykov}
\address{Department of
Mathematics \\ University of Nebraska-Lincoln \\
Lincoln, NE  68588-0130, USA} \email{dtoundykov@unl.edu}

\title[Hyperbolic equations of viscoelasticity]
{Hadamard well-posedness for a hyperbolic equation of viscoelasticity with supercritical sources and damping}
\date{April 14, 2014. Appeared in: Journal of Differential Equations 257 (2014), 3778--3812. }

\begin{document}
\maketitle

\begin{abstract}
Presented here is a study of a viscoelastic wave equation with supercritical source and damping terms. We employ the theory of monotone operators and nonlinear semigroups, combined with energy methods to establish the existence of a unique local weak solution. In addition, it is shown that the solution depends continuously on the initial data and is global provided the damping dominates the source in an appropriate sense.
\end{abstract}

\section {Introduction }\label{S1}

\subsection{Literature Overview}
The theory of viscoelasticity encompasses description of materials that exhibit a combination of elastic (able to recover the original shape after stress application) and viscous (deformation-preserving after stress removal) characteristics. Quantitative description of such substances involves a  strain-stress relation that depends on time. The classical linearized model yields an integro-differential equation that augments the associated elastic stress tensor with an appropriate \emph{memory term} which encodes the history of the deformation gradient.  The foundations of the theory go back to pioneering works of Boltzmann \cite{Boltz1}. For fundamental modeling developments see \cite{Col1} and the references therein.

When considering propagation of sound waves in viscoelastic fluids, if we neglect shear stresses then the stress tensor field may be expressed in terms of the acoustic pressure alone \cite{PZT}. Thereby we obtain a scalar wave equation  with a memory integral. This simplified formulation in fact captures most of the critical difficulties associated with the
well-posedness of the viscoelastic vectorial model \cite{Col1, Pata}, and therefore the multi-dimensional scalar wave equation with memory will be the subject of the subsequent discussion.

One can consider such an integro-differential equation with a finite or infinite time delay (the former being a special case of the infinite-delay where the strain is zero for all $t<0$). When restricted to the finite memory setting the system does not generate a semigroup, whereas the infinite-delay model can be represented as a semigroup evolution with the help of an appropriately defined history variable.

The (linear) viscoelastic problem with infinite memory  and its stability were extensively addressed in  \cite{Daf2,Daf1,Fab1}.  Existence of global attractors for wave equations with infinite memory in the presence of  nonlinear sources and linear internal damping (velocity feedback) was investigated in \cite{Pata}. The ``source"  here refers to amplitude-dependent feedback nonlinearity  whose growth rate is polynomially bounded with exponent $p\geq 1$. Due to the regularity of finite-energy solutions for this problem--- $H^1$ Sobolev level for the displacement variable---the source considered in the latter reference was \emph{subcritical} ($p< n/(n-2)$ for dimensions $n>2$) with respect to this energy.
Subsequently in \cite{PPZ} the authors look at attractors for the problem with strong (Kelvin-Voigt) damping and higher-order sources, including exponents of maximal order for which the associated energy is defined ($p=5$ in 3D).

A larger body of work is available on the finite-time delay problem.    The papers in this  list focus predominantly on well-posedness and asymptotic stability with energy dissipation due to memory and interior and/or boundary velocity feedbacks.  In addition, the sources, if present  are at most critical, i.e, $p\leq n/(n-2)$ in dimensions above $2$. See \cite{Cav3} for a treatment of interior and boundary memory  with nonlinear boundary damping and no sources.  Energy decay was investigated under localized interior dissipation and a source term was addressed in \cite{Cav4,Cav5}. Local and global well-posedness with source, but now without additional frictional damping was the subject of the paper \cite{BMess}. For systems of coupled waves with memory see
\cite{Han1}. Recent blow-up results for viscoelastic wave equations can be found in \cite{Han2,Liu1}. For quasilinear viscoelastic models with no sources and Kelvin-Voight damping refer for example to \cite{CC,Mess1}.

Overall, it appears that the finite-time memory case has been more actively studied. Yet to our knowledge presently there are no works dealing with super-critical source exponents (i.e.,  $p >3$  in 3D) in combination with memory terms. In light of this trend the present goal of this paper is two-fold:
\begin{itemize}
\item  Analyze the viscoelastic wave equation with sources beyond the critical level---so the potential energy is  no longer defined,---for instance, when $p>5$ in 3-space dimensions.  Our study is inspired by the advances in  \cite{B1,BL3,BL2, BL1}  and the consequent developments in \cite{BRT,GR1,GR2,GR,RW}.

\item Provide a treatment of this problem in the context of evolution semigroup framework with a self-contained detailed description of the generator and function spaces.

\end{itemize}

\subsection{The model}

Throughout, $\O\subset \reals^3$ is a bounded domain (open and connected) with boundary $\G$ of class $C^2$. Our results extend easily to bounded domains in $\reals^n$, by accounting for the corresponding Sobolev embeddings, and accordingly adjusting the conditions imposed on the parameters. In this paper, we study the following model:
\begin{align} \label{1.1}
\begin{cases}
u_{tt}- k(0) \Delta u -  \int_0^{\infty} k'(s) \Delta u(t-s) ds + g(u_t)=f(u),  \quad \text{ in } \O \times (0,\infty), \\
u(x,t)=0, \quad \text{ on }  \Gamma \times \reals, \\
u(x,t)=u_0(x,t), \quad \text{ in } \O \times  (-\infty,0],
\end{cases}
\end{align}
where, as mentioned earlier,  the unknown $u$ is an $\reals$-valued function on $\Omega\times (0,\infty)$, which can be thought of as the acoustic pressure  of sound waves in viscoelastic fluids.
The differentiable scalar map $k$ satisfies: $k(0), k(\infty) >0$ with $k'(s)< 0$ for all $s>0$. Here, $g$ is a monotone feedback, and $f (u)$  is a source. The memory integral
\[
\int_0^{\infty} k'(s) \Delta u(t-s) ds
\]  quantifies
the viscous resistance and provides a weak form of energy dissipation.  It also accounts for the full past history as time goes to $-\infty$, as opposed to the finite-memory model where the history is taken only over the interval $[0,t]$.

A similar model to (\ref{1.1}) was studied in \cite{Pata}, but with a linear interior damping and a source of a {\it dissipative sign} which is at most sub-critical. In our model (\ref{1.1}), the power-type damping $g(u_t)$ is nonlinear and not under any growth restrictions at the origin or at infinity; while the \emph{energy building} source $f(u)$ is possibly of \emph{supercritical} order.

For simplicity,  we set $\mu(s)=-k'(s)$ and $k(\infty)=1$, and so $k(0)>1$. Thus,   $\mu: \reals^+\longrightarrow \reals^+$,  where in Assumption \ref{ass} below precise assumptions on $\mu$ will be imposed.

\subsection{Notation}
The following notations will be used throughout the paper:
\begin{align}\label{1.2}
\norm{u}_s=\norm{u}_{L^s(\O)}; \;\;\;\;\langle u,v \rangle=\langle u,v \rangle_{H^{-1}(\O)\times H^1_0(\O)}; \;\;\;\; (u,v)=(u,v)_{L^2(\O)}.
\end{align}
The inner product on the weighted  the Hilbert space $L^2_{\mu}(\reals^+,H^1_0(\O))$ is defined by
\begin{align} \label{inner}
(u,v)_{\mu}:=\int_0^{\infty}\int_{\O}\nabla u(s) \cdot \nabla v(s) dx \, \mu(s) ds.
\end{align}
Also, $\norm{\cdot}_{\mu}$ represents the norm in $L^2_{\mu}(\reals^+,H^1_0(\O))$.
Subsequently,  we have:
 \begin{align} \label{H^1}
 H^1_{\mu}(\reals^+,H^1_0(\O)) =\{u\in L^2_{\mu}(\reals^+,H^1_0(\O)): u_t\in L^2_{\mu}(\reals^+,H^1_0(\O))\}.
 \end{align}
In particular, the space
$L^2_{\mu}(\reals^-,H^1_0(\O))$ consists of all functions $u: (-\infty,0]\longrightarrow H^1_0(\O)$ such that $u(-t)\in L^2_{\mu}(\reals^+,H^1_0(\O))$. Thus,
$$H^1_{\mu}(\reals^-,H^1_0(\O))=\{u\in L^2_{\mu}(\reals^-,H^1_0(\O)): u_t\in L^2_{\mu}(\reals^-,H^1_0(\O))\}.$$

\subsection{Main Results}
In light of the above discussion,  the following assumptions will be imposed throughout the paper.
\begin{assumption} \label{ass}\leavevmode
\begin{itemize}
\item ~~ $g$ is a continuous and monotone increasing  feedback with $g(0)=0$. In addition, the following growth condition at infinity holds:
there exist positive constants $a$ and $b$ such that, for $|s|\geq 1$,
$$a|s|^{m+1}\leq g(s)s\leq b|s|^{m+1}, \text{where\;\;} m\geq 1;$$
\item ~~$f\in C^1(\reals)$ such that $|f'(s)|\leq C(|s|^{p-1}+1)$, $1\leq p<6$;
\item ~~ $p\frac{m+1}{m}<6$;
\item ~~ $\mu \in C^1(\reals^+)\cap L^1(\reals^+)$ such that $\mu(s)> 0$ and $\mu'(s)\leq 0$ for all $s>0$, and $\mu(\infty)=0$;
\item ~~ $u_0(x,t)\in L^2_{\mu}(\reals^-,H^1_0(\O))$ with $\partial_t u_0(x,t)\in L^2_{\mu}(\reals^-,L^2(\O))$ such that $u_0: \mathbb R^- \rightarrow H^1_0(\O)$ and
$\partial_t u_0(x,t): \mathbb R^- \rightarrow L^2(\O)$ are weakly continuous at $t=0$. In addition, for all $t\leq 0$, $u_0(x,t)=0$ on $\G$.
\end{itemize}
\end{assumption}

Let us note here that in view  of the Sobolev imbedding $H^1(\O)  \hookrightarrow  L^6(\O) $ (in 3D),  the Nemytski operator $f(u)$ is locally Lipschitz continuous  from $ H^1_0(\O) $ into $L^2(\O) $ for the values $1\leq p \leq 3$. Hence, when the exponent of the sources $p$ lies in $1\leq p<3$, we call the source \emph{sub-critical}, and  \emph{critical}, if $p=3$.
 For the values $3< p\leq 5$  the source is  called \emph{supercritical}, and in this case the operator  $f(u)$ is not locally Lipschitz continuous  from $H^1_0(\O)$ into $L^2(\O) $. When  $5< p<6  $ the source is called \emph{super-supercritical}. In this case, the potential energy may not be defined in the finite energy space and the problem itself is no longer within the framework of potential well theory.

Recently, the boundary value problem for the wave equation with nonlinear damping and supercritical source (but without the memory term):
\begin{align*}
\begin{cases}
u_{tt}-\Delta u+g(u_t)=f(u), \text{\;\;in\;\;} \O \times (0,\infty),\\
\partial_{\nu}u+u+g_0(u_t)=h(u), \text{\;\;on\;\;} \Gamma \times (0,\infty),
\end{cases}
\end{align*}
has been studied in a series of papers \cite{BRT,B1,BL3,BL2, BL1}. One may see \cite{BRT-IMACS} for a summary of these results. Also, for other related work on nonlinear wave equations with supercritical sources, we refer the reader to \cite{Guo,GR1,GR2,GR,RTW,RS2,RW}.

It should be mentioned here  that (\ref{1.1}) is a monotonic problem well-suited for utilizing the theory of nonlinear semigroups and monotone operators (see for instance \cite{Barbu3,Sh}). Thus, for the local well-posedness of (\ref{1.1}), our strategy draws substantially from ideas in \cite{B1,BL1,GR,RW}. The essence of this strategy is to
write the problem as a Cauchy problem of semigroup form and
set up an appropriate phase space in order to verify the semigroup generator is m-accretive. The difficulty lies in the justification of the maximal monotonicity and coercivity of a certain nonlinear operator, which requires a correct choice of the function space and a combination of various techniques in monotone operator theory.

In order to state our main results, we begin  with the definition of a weak solution of  (\ref{1.1}).
\begin{definition} \label{def-weak}
A function $u(x,t)$ is said to be a \emph{weak solution} of (\ref{1.1}) on $(-\infty,T]$
if $u\in L^2_{\mu}((-\infty,T];H^1_0(\O)) \cap C([0,T];H^1_0(\O))$ such that $u_t\in L^2_{\mu}((-\infty,T];L^2(\O)) \cap C([0,T];L^2(\O))\cap L^{m+1}(\O \times (0,T))$ with:
\begin{itemize}
\item $u(x,t)=u_0(x,t)$ for $t\leq 0$;
\item The following variational identity holds for all $t\in [0,T]$, and all test functions $\phi \in \mathscr{F}$:  \begin{align} \label{weak}
&(u_t(t),\phi(t))-(u_t(0),\phi(0))-\int_0^t \int_{\O}u_t(\t)\phi_t(\t)dx d\t\notag\\
&+k(0)\int_0^t \int_{\O} \nabla u(\t) \cdot \nabla \phi(\t) dx d\t
+\int_0^t \int_0^{\infty} \int_{\O} \nabla u(\t-s) \cdot \nabla \phi(\t) dx k'(s) ds d\t\notag\\
&+\int_0^t \int_{\O} g(u_t(\t))\phi(\t) dx d\t
=\int_0^t \int_{\O} f(u(\t)) \phi(\t) dx d\t,
\end{align}
where $$ \mathscr{F} = \Big\{ \phi: \,\, \phi \in C([0,T];H^1_0(\O))\cap L^{m+1}(\O \times (0,T)) \text{  with  } \phi_t\in C([0,T];L^2(\O)) \Big\}.$$
\end{itemize}
\end{definition}

Our first theorem gives the existence and uniqueness of local weak solutions.

\begin{theorem}[{\bf Short-time existence}] \label{thm-exist}
Assume the validity of the Assumption \ref{ass}, then there exists a local (in time) weak solution $u$ to (\ref{1.1}) defined on $(-\infty,T]$ for some $T>0$ depending on the initial energy $E(0)$. Furthermore, the following energy identity holds:
\begin{align} \label{EI-0}
&E(t)+\int_0^t \int_{\O}g(u_t)u_t dx d\t-\frac{1}{2}\int_0^t \int_0^{\infty} \norm{\nabla w(\t, s)}_2^2 \mu'(s)ds d\t \notag\\
&=E(0)+\int_0^t \int_{\O} f(u)u_t dx d\t,
\end{align}
where $w(x,\t,s)=u(x,\t)-u(x,\t-s)$, and the quadratic energy is defined by
\begin{align} \label{energy}
E(t)=\frac{1}{2}\left(\norm{u_t(t)}_2^2+\norm{\nabla u(t)}_2^2+\int_0^{\infty}\norm{\nabla w(t,s)}_2^2 \mu(s)ds \right).
\end{align}
\end{theorem}

Our next result states that weak solutions of (\ref{1.1}) depend continuously on the initial data.
\begin{theorem}[{\bf Continuous dependence on initial data}]  \label{thm-cont}
In addition to the Assumption \ref{ass}, assume that  $u_0(0)\in L^{\frac{3(p-1)}{2}}(\O)$ and $f\in C^2(\reals)$ such that
$|f''(s)|\leq C(|s|^{p-2}+1)$, for $p>3$.
If $u_0^n\in L^2_{\mu}(\reals^-, H^1_0(\O))$ is a sequence of initial data such that $u_0^n\longrightarrow u_0$ in $L^2_{\mu}(\reals^-,H^1_0(\O))$ with
$u^n_0(0)\longrightarrow u_0(0)$ in $H^1_0(\O)$ and in $L^{\frac{3(p-1)}{2}}(\O)$, $\frac{d}{dt}u^n_0(0)\longrightarrow \frac{d}{dt}u_0(0)$ in $L^2(\O)$, then the corresponding weak solutions $u_n$ and $u$ of (\ref{1.1}) satisfy
\begin{align*}
u_n\longrightarrow u \text{\;\;in\;\;}  C([0,T];H^1(\O)) \text{\;\;\;and\;\;\;}
u_n'\longrightarrow u'  \text{\;\;in\;\;}  C([0,T];L^2(\O)).
\end{align*}
\end{theorem}

The uniqueness of weak solutions is a corollary of Theorem \ref{thm-cont}.
\begin{corollary}[{\bf Uniqueness}]  \label{thm-unique}
In addition to the Assumption \ref{ass}, we assume $u_0(0)\in L^{\frac{3(p-1)}{2}}(\O)$ and $f\in C^2(\reals)$ such that
$|f''(s)|\leq C(|s|^{p-2}+1)$, for $p>3$. Then, weak solutions of (\ref{1.1}) are unique.
\end{corollary}

Our final result states: if the damping dominates the source term, then the solution is global. More precisely, we have
\begin{theorem}[{\bf Global existence}] \label{thm-global}
In addition to Assumption \ref{ass}, further assume $u_0(0) \in L^{p+1}(\O)$. If $m\geq p$, then the weak solution of (\ref{1.1}) is global.
\end{theorem}

\begin{remark}\label{rem-source}
The classical condition that the ``damping dominates the source," $m\geq p$, in Theorem \ref{thm-global} can be dispensed with if the source $f(u)$ in the equation satisfies suitable dissipativity conditions. For example, if the scalar function $f$ is monotone decreasing with $f(s)s \leq 0$ for all $s\in \mathbb{R}$, then the assumption $m\geq p$ can be removed.
\end{remark}

\section{Local solutions}

This section is devoted to prove the local existence statement in Theorem \ref{thm-exist}.

\subsection{Operator Theoretic Formulation}

In order to study the local solvability of (\ref{1.1}), we exploit a remarkable  idea due to Dafermos \cite{Daf2,Daf1}: in addition to the displacement and velocity, we regard the past history of the displacement as a third variable. More precisely, introduce the {\em history function}:
\begin{align} \label{Def-w}
w(x,t,s)= u(x,t)-u(x,t-s), \;\;s\geq 0.
\end{align}
After  simple manipulations,  problem (\ref{1.1}) can be put into the following coupled system:
\begin{align} \label{2.2}
\begin{cases}
 u_t(x,t)=v(x,t) \\
 v_{t}(x,t)=\Delta u(x,t) + \int_0^{\infty} \mu(s) \Delta w(x,t,s) ds  - g(v(x,t)) + f(u(x,t))\\
 w_t(x,t,s)=v(x,t)-w_s(x,t,s),
\end{cases}
\end{align}
with boundary and initial conditions
\begin{align}
\begin{cases}
 u(x,t)=0 \quad \text{on }  \Gamma \times [0,\infty),\\
 w(x,t,s)=0 \quad \text{on }  \Gamma \times [0,\infty) \times [0,\infty) \\
 u(x,0)=u_0(x,0)\\
 v(x,0)=\frac{\partial u_0}{\partial t}(x,0)\\
 w(x,0,s)=u_0(x,0)-u_0(x,-s)\\
 w(x,t,0)=0.
\end{cases}
\end{align}

\begin{remark}  \label{remark1}
System (\ref{2.2}) with the given initial and boundary conditions is equivalent to the original system (\ref{1.1}). In fact, by using the method of characteristics (in the $(t,s)$-plane where $x$ is regarded as a fixed parameter), one can see that the equations $w_t=v-w_s$ and $u_t=v$ in (\ref{2.2}) with the condition $w(x,t,0)=0$ imply $w(x,t,s)= u(x,t)-u(x,t-s)$ for $s\geq 0$.
\end{remark}

We establish the local in time existence of weak solutions in the so called \emph{past history framework}, i.e., the unknown function $(u,v,w)$ is in the phase space $$H:=H^1_0(\O)\times L^2(\O)\times L^2_\mu(\reals^+,H^1_0(\O)).$$
If $U=(u,v,w), \, \hat U=(\hat u,\hat v,\hat w)\in H$, then the inner product on the Hilbert space $H$ is
 the natural inner product given by:
$$(U, \hat U)_H:= \int_\O \nabla u \cdot \nabla \hat  u \, dx + (v,\hat v) + (w, \hat w)_{\mu},$$
where $(v,\hat v)$ and  $(w, \hat w)_{\mu},$ are given in (\ref{1.2})-(\ref{inner}).

If $\xi \in L^2_\mu(\reals^+,H^2(\O)\cap H^1_0(\O))$, then clearly $\int_0^{\infty}\mu(s)\Delta \xi(s)ds\in L^2(\O)\subset H^{-1}(\O)$. Thus,  for all $\p\in H^1_0(\O)$, we have
\begin{align} \label{e0}
\left \langle \int_0^{\infty}\mu(s)\Delta \xi (s)ds, \p \right \rangle &=\int_{\O}\left(\int_0^{\infty}\mu(s)\Delta  \xi (s)ds\right) \p dx \notag\\
&=-\int_0^{\infty}\int_{\O}\nabla  \xi (s)\cdot \nabla \p dx \mu(s) ds=-( \xi ,\p)_{\mu},
\end{align}
where  $( \xi ,\phi)_{\mu}$ is defined in (\ref{inner}).

Now, we define the operator $\mathcal L: D(\mathcal L)\subset L^2_{\mu}(\reals^+, H^1_0(\O))\longrightarrow H^{-1}(\O)$ by
$$\mathcal L( \xi)=\int_0^{\infty}\mu(s)\Delta  \xi (s)ds$$
where $D(\mathcal L)=L^2_{\mu}(\reals^+, H^2(\O)\cap H^1_0(\O))$. It follows from (\ref{e0}) that
$\langle \mathcal L( \xi ),\phi \rangle=-( \xi ,\phi)_{\mu}$ for all $\p\in H^1_0(\O)$.
Clearly,  $\mathcal L$ is a linear mapping, and in addition, $\mathcal L$ is bounded from $D(\mathcal L)$ into $ H^{-1}(\O)$. Indeed,    for  all  $ \xi \in D(\mathcal L)$, we have
\begin{align*}
\norm{\mathcal L( \xi )}_{H^{-1}(\O)}=\sup_{\norm{\p}_{H_0^1(\O)}=1}|( \xi ,\p)_{\mu}|\leq \norm{ \xi}_{\mu}\left(\int_0^{\infty}\mu(s)ds\right)^{\frac{1}{2}}=\norm{ \xi }_{\mu}(k(0)-1).
\end{align*}
Therefore,  we can extend $\mathcal L$ to be a bounded linear operator (which is still denoted by $\mathcal L$)  from $L^2_{\mu}(\reals^+, H^1_0(\O))$ to $H^{-1}(\O)$ such that, for any $ \xi \in L^2_\mu(\reals^+,H^1_0(\O))$,
\begin{align} \label{duality}
\left \langle \mathcal L( \xi), \p \right \rangle=-( \xi ,\p)_{\mu}
\end{align}
for all $\phi\in H^1_0(\O)$.

To this end, we define an (abstract)  operator $\mathscr{A}: D(\mathscr A)\subset H\longrightarrow H$ by
\[\mathscr A\left(U\right)=\begin{pmatrix}
-v \\
-\Delta u+g(v)- \mathcal L(w)-f(u) \\
-v+w_s
\end{pmatrix} ^{tr}\]
with its domain
\begin{align*}
D(\mathscr A)=\{&(u,v,w)\in H^1_0(\O)\times H^1_0(\O) \times H^1_{\mu}(\reals^+,H^1_0(\O)):
g(v)\in H^{-1}(\O)\cap L^1(\O), \notag\\
&-\Delta u+g(v)-\mathcal L(w)-f(u) \in L^2(\O), \; w(0)=0\}.
\end{align*}
Since the original  $w$ is a  function of the three variables $(x,t,s)$, then, in the definition of the operator $\mathscr A$ above,  by saying $ w\in H^1_{\mu}(\reals^+,H^1_0(\O))$ we only mean the mapping: $\reals^+ \ni s\longmapsto w(\cdot ,s)$  belongs to $ H^1_{\mu}(\reals^+,H^1_0(\O))$,  as defined in (\ref{H^1}).

Henceforth, system (\ref{2.2}) can be reduced to the Cauchy problem:
\begin{align} \label{semigroup}
\begin{cases}
U_t+\mathscr{A}U=0,\\
U(0)=U_0=\Big(u_0(x,0),\partial_t u_0(x,0),u_0(x,0)-u_0(x,-s)\Big).
\end{cases}
\end{align}

\subsection{Globally Lipschitz Source}

Our first proposition gives the existence of a global solution to the Cauchy problem (\ref{semigroup}) provided the source $f$ is globally Lipschitz from $H^1_0(\O)$ to $L^2(\O)$.

\begin{proposition} \label{lemma-1}
Assume $g$ is a continuous and monotone increasing function such that $g(0)=0$. In addition, assume $f: H_0^1(\O)\longrightarrow L^2(\O)$ is globally Lipschitz continuous. Then, system (\ref{semigroup}) has a unique global strong solution $U\in W^{1,\infty}(0,T;H)$ for arbitrary $T>0$ provided the initial datum $U_0\in D(\mathscr A)$.
\end{proposition}

\begin{proof}
In order to prove Proposition \ref{lemma-1}, it suffices to show that the operator $\mathscr A+\a I$ is $m$-accretive for some positive $\a$. We say an operator
$\mathscr A: \mathcal D(\mathscr A) \subset H \longrightarrow H$ is \emph{accretive} if
$(\mathscr Ax_1-\mathscr Ax_2,x_1-x_2)_H \geq 0$, for all $x_1,x_2 \in \mathcal D(\mathscr A)$,
and it is \emph{$m$-accretive} if, in addition, $\mathscr A+I$ maps $\mathcal D(\mathscr A)$ onto $H$.
It follows from Kato's Theorem (see \cite{Sh} for instance) that, if $\mathscr A+\a I$ is $m$-accretive for some positive $\a$, then for each $U_0 \in \mathcal D(\mathscr A)$ there is a unique global strong solution $U$ of the Cauchy problem (\ref{semigroup}).\\

\noindent
\textbf{Step 1}: We show that $\mathscr{A}+\a I: D(\mathscr A)\subset H\longrightarrow H$ is an accretive operator for some $\a>0$. Let $U=(u,v,w)$, $\hat U=(\hat u,\hat v,\hat w)\in D(\mathscr A)$.  For sake of simplifying the notation in this proof, we use the notation $\innerprod{\cdot,\cdot}$ to denote  the standard duality pairing between
$H^{-1}(\O)$ and $H^{1}_0(\O)$; while $(\cdot,\cdot)$ represents the inner product in $L^2(\O)$.

Then,
\begin{align} \label{e1}
&((\mathscr A+\a I)U-(\mathscr A+\a I)\hat U, U-\hat U)_H=(\mathscr A(U)-\mathscr A(\hat U),U-\hat U)_H+\a\|U-\hat U\|^2_H \notag \\
&=-(\nabla(v-\hat v),\nabla(u-\hat u))-\langle \Delta(u-\hat u),v-\hat v \rangle+\langle g(v)-g(\hat v),v-\hat v \rangle \notag\\
&\;\;\;\;-\left \langle \mathcal L(w-\hat w), v-\hat v \right \rangle
-(f(u)-f(\hat u),v-\hat v) \notag\\
&\;\;\;\;-(v-\hat v,w-\hat w)_{\mu}+(w_s-\hat w_s,w-\hat w)_{\mu}+\a \|U-\hat U\|^2_H.
\end{align}

First, thanks to  (\ref{duality}), we have
\begin{align} \label{e2}
-\left \langle \mathcal L(w-\hat w), v-\hat v \right \rangle=(w-\hat w,v-\hat v)_{\mu}.
\end{align}

Since $U$ and $\hat U\in D(\mathscr A)$, we know $g(v)-g(\hat v)\in H^{-1}(\O)\cap L^1(\O)$. Thus, by the monotonicity of $g$ and Lemma 2.6 in \cite{Barbu3},
one has $\left(g(v)-g(\hat v)\right)\left(v-\hat v\right)\in L^1(\O)$ and
\begin{align} \label{e3}
\innerprod{g(v)-g(\hat v),v-\hat v}
= \int_\O\left(g(v)-g(\hat v)\right)\left(v-\hat v\right)dx\geq 0.
\end{align}

Since $w-\hat w \in  H^1_{\mu}(\reals^+,H^1_0(\O))$, then by virtue of (\ref{inner}),
\begin{align} \label{e4}
(w_s-\hat w_s,    w-\hat w)_{\mu}&=\frac{1}{2}\int_0^{\infty}\frac{d}{ds}\left(\int_{\O}|\nabla(w(s)-\hat w (s))|^2 dx \right)\mu(s)ds \notag\\
&=-\frac{1}{2}\int_0^{\infty}\left(\int_{\O}|\nabla(w(s)-\hat w (s))|^2 dx \right)\mu'(s)ds \geq 0,
\end{align}
where we have used integration by parts and the facts: $\mu(\infty)=0$, $\mu'(s)\leq 0$ and $w(0)=0$.

Since $f$ is globally Lipschitz continuous from $H^1_0(\O)$ into $L^2(\O)$ with Lipschitz constant $L_f$, it follows that
\begin{align} \label{e5}
(f(u)-f(\hat u),v-\hat v)&\leq L_f\norm{\nabla(u-\hat u)}_2\norm{v-\hat v}_2 \notag\\
& \leq \frac{L_f}{2}\left(\norm{\nabla(u-\hat u)}^2_2+\norm{v-\hat v}^2_2\right).
\end{align}

Therefore,  (\ref{e1})-(\ref{e5}) yield
\begin{align} \label{e6}
((\mathscr{A}+kI)U & - (\mathscr{A}+kI)\hat U,  U-\hat U)_H  \notag\\
&\geq -\frac{L_f}{2}\left(\norm{u-\hat u}^2_2+\norm{v-\hat v}^2_2\right)+\a\|U-\hat U\|^2_H  \notag\\ &
\geq (\a-\frac{L_f}{2})\|U-\hat U\|^2_H \geq 0;
\end{align}
provided $\a \geq \frac{L_f}{2}$.\\


\noindent
\textbf{Step 2}: We show that $\mathscr{A}+\l I$ is m-accretive for some $\l>0$. To this end, it suffices to show that the range of $\mathscr{A}+\l I$ is all of $H$, for some $\l>0$ (see for example \cite{Sh}).

Let $(a,b,c)\in H$. We aim to show that there exists $(u,v,w)\in D(\mathscr A)$ such that $(\mathscr A+\l I)(u,v,w)=(a,b,c)$, for some $\l>0$, i.e.,
\begin{align} \label{e9}
\begin{cases}
-v+\l u =a \\
-\Delta u+g(v)- \mathcal L(w)-f(u) +\l v=b\\
-v+w_s +\l w=c.
\end{cases}
\end{align}

Notice that, (\ref{e9}) is equivalent to
\begin{align} \label{e10}
\begin{cases}
-\frac{1}{\l}\Delta v+g(v)- \mathcal L(w)-f(\frac{v+a}{\l}) +\l v=b+\frac{1}{\l}\Delta a\\
-v+w_s +\l w=c.
\end{cases}
\end{align}

Let $X=H^1_0(\O)\times L^2_{\mu}(\reals^+, H_0^1(\O))$ where $X$ is endowed with the natural inner product, i.e., if $U=(v,w), \, \hat U=(\hat v, \hat w)\in X$, then
$$(U, \hat U)_X:= \int_\O \nabla v \cdot \nabla \hat  v \, dx +  (w, \hat w)_{\mu}.$$

Define an operator $T: D(T)\subset X \longrightarrow X'$ by
\begin{align*}
T\begin{pmatrix}
v\\
w
\end{pmatrix}^{tr}=\begin{pmatrix}
-\frac{1}{\l}\Delta v+g(v)- \mathcal L(w)-f(\frac{v+a}{\l}) +\l v\\
-v+w_s +\l w
\end{pmatrix}^{tr}
\end{align*}
where,
$$D(T)=\Big\{(v,w)\in X: g(v)\in H^{-1}(\O)\cap L^1(\O), \;w\in H_{\mu}^1(\reals^+, H_0^1(\O)), \; w(0)=0 \Big\}.$$
It is important to note here  that we consider $L^2_{\mu}(\reals^+,H^1_0(\O))$ as a Hilbert space identified with its own dual, and thus,
$X'=H^{-1}(\O)\times L^2_{\mu}(\reals^+, H_0^1(\O))$.

To justify the surjectivity of $T$, it is sufficient to show that the operator $T$ is coercive and maximal monotone (Corollary 2.2  in \cite{Barbu3}).

We split $T$ as a summation of three operators. First, we define $T_1: X\longrightarrow X'$ by
\begin{align*}
T_1\begin{pmatrix}
v\\
w
\end{pmatrix}^{tr}=\begin{pmatrix}
-\frac{1}{\l}\Delta v- \mathcal L(w)-f(\frac{v+a}{\l}) +\l v\\
-v+\l w
\end{pmatrix}^{tr}.
\end{align*}
The operator $T_2: D(T_2)\subset H^1_0(\O) \longrightarrow H^{-1}(\O)$ is defined by
\begin{align*}
T_2(v)=g(v)
\end{align*}
where $D(T_2)=\{v\in H^1_0(\O): g(v)\in H^{-1}(\O)\cap L^1(\O)\}$. By a result due to Br\'ezis  \cite{Bre}, $T_2$ is the
sub-differential of the convex functional $J: H^1_0(\O)\longrightarrow [0,\infty]$ defined by $J(u)=\int_{\O}j(u)dx$, where $j(s)=\int_0^s g(\t) d\t$. It is well-known that the subdifferential of a proper convex function is maximal monotone, and thus $T_2$ is a maximal monotone operator.

We further define $T_3: D(T_3) \subset L^2_{\mu}(\reals^+,H^1_0(\O)) \longrightarrow L^2_{\mu}(\reals^+,H^1_0(\O))$ by
\begin{align*}
T_3(w)=\partial_sw
\end{align*}
where $D(T_3)=\{w\in H^1_{\mu}(\reals^+,H^1_0(\O)): w(0)=0\}$.
Notice that the monotonicity of   $T_3$ follows from (\ref{e4}). In addition, it is clear that the operator
 $T_3+I$ is surjective. Therefore, $T_3$ is maximal monotone (see Theorem 2.2 in \cite{Barbu3}).

To see $T_1$ is maximal monotone from $X$ to $X'$, it is enough to verify that $T_1$ is monotone and hemicontinuous. For checking the monotonicity of $T_1$, we consider
\begin{align} \label{e11}
&\left \langle T_1\begin{pmatrix}v\\w\end{pmatrix}^{tr}-T_1\begin{pmatrix}\hat v\\ \hat w\end{pmatrix}^{tr},
\begin{pmatrix}v\\w\end{pmatrix}^{tr}-\begin{pmatrix}\hat v\\ \hat w\end{pmatrix}^{tr} \right \rangle_{X'\times X} \notag\\
&=\frac{1}{\l}\norm{\nabla(v-\hat v)}^2_2+(w-\hat w,v-\hat v)_{\mu} -\left(f\Big(\frac{v+a}{\l}\Big)-f\Big(\frac{\hat v+a}{\l}\Big),v-\hat v\right) \notag\\
&+\l\norm{v-\hat v}^2_2
-(v-\hat v,w-\hat w)_{\mu}+\l\norm{w-\hat w}^2_{\mu}\notag\\
&=\frac{1}{\l}\norm{\nabla(v-\hat v)}^2_2+\l\norm{v-\hat v}^2_2-\left(f\Big(\frac{v+a}{\l}\Big)-f\Big(\frac{\hat v+a}{\l}\Big),v-\hat v\right)\notag\\
&+\l\norm{w-\hat w}^2_{\mu}.
\end{align}
Since $f$ is globally Lipschitz continuous from $H^1_0(\O)$ into $L^2(\O)$ with Lipschitz constant $L_f$, one has
\begin{align} \label{e13}
&\left(f\Big(\frac{v+a}{\l}\Big)-f\Big(\frac{\hat v+a}{\l}\Big),v-\hat v\right)\leq
\norm{f\Big(\frac{v+a}{\l}\Big)-f\Big(\frac{\hat v+a}{\l}\Big)}_2\norm{v-\hat v}_2 \notag\\
&\leq \frac{L_f}{\l}\norm{\nabla(v-\hat v)}_2 \norm{v-\hat v}_2
\leq \frac{1}{2\l^2}\norm{\nabla(v-\hat v)}_2^2+\frac{1}{2}L_f^2\norm{v-\hat v}_2^2.
\end{align}
Combining (\ref{e11}) and (\ref{e13}) gives
\begin{align}\label{2.17}
&\left \langle T_1\begin{pmatrix}v\\w\end{pmatrix}^{tr}-T_1\begin{pmatrix}\hat v\\ \hat w\end{pmatrix}^{tr},
\begin{pmatrix}v\\w\end{pmatrix}^{tr}-\begin{pmatrix}\hat v\\ \hat w\end{pmatrix}^{tr} \right \rangle_{X'\times X} \notag\\
&\geq \left(\frac{1}{\l}-\frac{1}{2\l^2}\right)\norm{\nabla(v-\hat v)}_2^2+\left(\l-\frac{1}{2}L_f^2\right)\norm{v-\hat v}_2^2
+\l\norm{w-\hat w}^2_{\mu}.
\end{align}
Thus, it follows from (\ref{2.17}) that  $T_1$ is strongly monotone;  provided $\l>\frac{1}{2}\max\{L_f^2,1\}$. Also, it is easy to see that strong monotonicity implies coercivity of $T_1$.

Next we verify $T_1: X\longrightarrow X'$ is hemicontinuous. Clearly, any linear operator is hemicontinuous. So, we merely consider the nonlinear term $f\left(\frac{v+a}{\l}\right)$. The fact that $f$ is globally Lipschitz from $H_0^1(\O)$ into $L^2(\O)$, trivially implies that $f\left(\frac{v+a}{\l}\right)$ is continuous from $H^1_0(\O)$ into $H^{-1}(\O)$. Hence,  $T_1$ is hemicontinuous, and so,  $T_1$ is maximal monotone.

Now, it is important to note that
\begin{align*}
T\begin{pmatrix}
v\\
w
\end{pmatrix}^{tr}=T_1\begin{pmatrix}
v\\
w
\end{pmatrix}^{tr}
+\begin{pmatrix}
T_2(v)\\
T_3(w)
\end{pmatrix}^{tr},
\end{align*}
where $T_2$ and $T_3$ are both maximal monotone which act on different components of the vector $\begin{pmatrix}v\\w\end{pmatrix}^{tr}$. By Proposition 7.1 in \cite{GR}, it follows that   the mapping $\begin{pmatrix}v\\w\end{pmatrix}^{tr} \mapsto \begin{pmatrix}T_2(v)\\T_3(w) \end{pmatrix}^{tr}$ is maximal monotone from $D(T)$ to $X'$. Moreover, due to the maximal monotonicity of $T_1$ and the fact that the domain of $T_1$ is the entire space $X$, we conclude that $T$ is maximal monotone (Theorem 2.6 in \cite{Barbu3}).

In addition, we know $T$ is coercive, since $T_1$ is coercive and both of $T_2$ and $T_3$ are monotone. Therefore, $T$ is maximal monotone and coercive, which yields the surjectivity of $T$, i.e, there exists $(v,w)\in D(T)$ satisfies (\ref{e10}) for any $(a,b,c)\in H$. By (\ref{e9}), $u=\frac{v+a}{\l}\in H^1_0(\O)$ and $-\Delta u+g(v)- \mathcal L(w)-f(u)=b-\l v\in L^2(\O)$. Consequently,  $(u,v,w)\in D(\mathscr A)$ which concludes the proof of Proposition \ref{lemma-1}.
\end{proof}

\subsection{Locally Lipschitz Source}

In this subsection, we loosen the restriction on the source by allowing $f$ to be locally Lipschitz continuous. More precisely, we have the following result.

\begin{proposition} \label{lemma-2}
Assume $g$ is a continuous and monotone increasing function vanishing at the origin such that $g(s)s\geq a|s|^{m+1}$ for all $|s|\geq 1$, where $a>0$ and $m\geq 1$.
In addition, assume $f: H_0^1(\O)\longrightarrow L^2(\O)$ is locally Lipschitz continuous. Then, system (\ref{semigroup}) has a unique local strong solution $U\in W^{1,\infty}(0,T;H)$, for some $T>0$, provided the initial datum $U_0\in D(\mathscr A)$.
\end{proposition}

\begin{proof}
We employ a standard truncation of the source. Define
\begin{align*}
f_K(u)=\left\{ \begin{array}{ll}
  f(u) & \mbox{if $\norm{\nabla u}_2 \leq K,$}\\
  f\left(\frac{Ku}{\norm{\nabla u}_2}\right)
  & \mbox{if $\norm{\nabla u}_2> K,$ }\\
                             \end{array}
                         \right.
\end{align*}
where $K$ is a positive constant. With this setting, we consider the following $K$-truncated problem:
\begin{align} \label{K}
U_t+\mathscr A_K U=0
\end{align}
with the same initial condition as in  problem (\ref{semigroup}), where the operator
$\mathscr{A}_K: D(\mathscr A_K)\subset H\longrightarrow H$ is defined by
\[\mathscr A_K\left(U\right)=\begin{pmatrix}
-v \\
-\Delta u+g(v)- \mathcal L(w)-f_K(u) \\
-v+w_s
\end{pmatrix}^{tr}\]
with its domain $D(\mathscr A_K)=D(\mathscr A)$.

Since the truncated source $f_K: H_0^1(\O)\longrightarrow L^2(\O)$ is globally Lipschitz continuous for each $K$ (see \cite{CEL1}), then by Proposition \ref{lemma-1}, the truncated  problem (\ref{K}) has a unique global strong solution $U_K\in W^{1,\infty}(0,T;H)$ for any $T>0$; provided the initial datum $U_0\in D(\mathscr A)$.

For simplifying the notation in the rest of the proof, we shall express $U_K$ as $U$. First, we aim to derive the associated  energy identity for (\ref{K}). Since $U=(u,v,w)$ is a strong solution of (\ref{K}), the following equation holds:
\begin{align} \label{e14}
v_t-\Delta u+g(v)-\mathcal L(w)-f_K(u)=0,\,\,\ \text{ a.e.  } [0,T].
\end{align}
By the regularity of the solution $U$, we can multiply (\ref{e14}) by $v=u_t$ and integrate on $\O \times (0,t)$ where $0<t<T$,  to obtain,
\begin{align} \label{e15}
&\frac{1}{2}(\norm{v(t)}_2^2+\norm{\nabla u(t)}^2_2)
+\int_0^t \int_{\O}g(v)v dx d\t+\int_0^t (w,v)_{\mu} d\t  \notag\\
&=\frac{1}{2}(\norm{v(0)}_2^2+\norm{\nabla u(0)}^2_2)+\int_0^t \int_{\O} f_K(u)v dx d\t,
\end{align}
where (\ref{e15}) hold for any $t>0$, as $T>0$ is arbitrary.

Since $v=w_t+w_s$, we compute
\begin{align} \label{e16}
&\int_0^t (w,v)_{\mu} d\t
=\int_0^t (w,w_t+w_s)_{\mu} d\t \notag\\
&=\int_0^t \int_0^{\infty} \int_{\O} \nabla w(\t,s) \cdot \nabla w_t(\t,s) dx \mu(s)ds d\t  \notag\\
&\hspace{0.4 in}+\int_0^t \int_0^{\infty} \int_{\O} \nabla w(\t,s) \cdot \nabla w_s(\t,s) dx \mu(s)ds d\t \notag\\
&=\frac{1}{2}\int_0^{\infty}\left(\norm{\nabla w(t,s)}_2^2-\norm{\nabla w(0,s)}_2^2 \right)  \mu(s)ds  \notag\\
&\hspace{0.4 in}-\frac{1}{2}\int_0^t \int_0^{\infty} \norm{\nabla w(\t,s)}_2^2 \mu'(s)ds d\t,
\end{align}
where we have used integration by parts with the fact $\mu(\infty)=0$ and $w(x,t,0)=0$.
Therefore,  (\ref{e15}) and (\ref{e16})  yield  the following energy identity:
\begin{align} \label{EI}
E(t) &+\int_0^t \int_{\O}g(v)v dx d\t-\frac{1}{2}\int_0^t \int_0^{\infty} \norm{\nabla w(\t,s)}_2^2 \mu'(s)ds d\t \notag\\
&=E(0)+\int_0^t \int_{\O} f_K(u)v dx d\t,
\end{align}
where the quadratic energy $E(t)$ is defined in (\ref{energy}).
Since $\mu'(s)\leq 0$, then for all $s>0$, we have
\begin{align} \label{energy-ineq}
E(t)+\int_0^t \int_{\O}g(v)v dx d\t \leq E(0)+\int_0^t \int_{\O} f_K(u)v dx d\t.
\end{align}

Let us note here that, straightforward calculation shows $f_K: H^1_0(\O) \longrightarrow L^{\frac{m+1}{m}}(\O)$ is globally Lipschitz with Lipschitz constant $L_K$ (see \cite{CEL1}).
Thus, we estimate term due to the source  on the right-hand side of the energy inequality (\ref{energy-ineq}) as follows:
\begin{align} \label{e17}
&\int_0^t \int_{\O} f_K(u)v dx d\t \leq  \int_0^t \norm{f_K(u)}_{\frac{m+1}{m}}\norm{v}_{m+1}d\t \notag\\
&\leq \e \int_0^t \norm{v}_{m+1}^{m+1} d\t+C_{\e}\int_0^t \norm{f_K(u)}_{\frac{m+1}{m}}^{\frac{m+1}{m}} d\t \notag\\
&\leq \e \int_0^t \norm{v}_{m+1}^{m+1} d\t+C_{\e}L_K^{\frac{m+1}{m}}\int_0^t \norm{\nabla u}_2^{\frac{m+1}{m}} d\t+C_{\e}t|f(0)|^{\frac{m+1}{m}}|\O|.
\end{align}

By recalling the assumption on the damping that  $g(s)s\geq a|s|^{m+1}$ for all $|s|\geq 1$, we have
\begin{align} \label{e18}
\int_0^t \int_{\O}g(v)v dx d\t \geq a \int_0^t \norm{v}_{m+1}^{m+1} d\t-at|\O|.
\end{align}

Thus,  (\ref{energy-ineq})-(\ref{e18}) and the fact  $\frac{m+1}{m}\leq 2$ yield,
\begin{align} \label{e20}
 &E(t)+a\int_0^t \norm{v}_{m+1}^{m+1} d\t \notag\\
 &\leq E(0)+\e\int_0^t \norm{v}_{m+1}^{m+1} d\t+C_{\e}L_K^{\frac{m+1}{m}}\int_0^t \norm{\nabla u}_{2}^{\frac{m+1}{m}}d\t+at|\O|+C_{\e}t|f(0)|^{\frac{m+1}{m}}|\O| \notag\\
&\leq E(0)+\e\int_0^t \norm{v}_{m+1}^{m+1} d\t+2C_{\e}L_K^{\frac{m+1}{m}}\int_0^t E(t) d\t+t C_0
\end{align}
where $C_0$ depends on $\e$, $L_K$, $f(0)$, $|\O|$ and $m$. By choosing $\e\leq a$ one has,
\begin{align*}
E(t)\leq E(0)+C_0 T+C(L_K)\int_0^t E(\t)d\t, \text{\;\;for all\;\;} t\in [0,T],
\end{align*}
where $C(L_K)=2C_{\e}L_K^{\frac{m+1}{m}}$, and $T$ will be chosen below.
By  Gronwall's inequality, one has
\begin{align*}
E(t)\leq (E(0)+C_0 T)e^{C(L_K)t}, \text{\;\;for all\;\;} t\in [0,T].
\end{align*}
If we select
\begin{align} \label{def-T}
T=\min\left\{\frac{1}{C_0},\frac{1}{C(L_K)}\log 2 \right\},
\end{align}
then
\begin{align} \label{e19}
E(t)\leq 2(E(0)+1)\leq K^2/2, \text{\;\;for all\;\;} t\in [0,T];
\end{align}
provided we choose
\begin{align} \label{bound-K}
K^2\geq 4(E(0)+1).
\end{align}
We note here that  (\ref{e19}) shows $\norm{\nabla u}_2\leq K$, for all
$t\in [0,T]$, and thus, by the definition of $f_K$, we see that $f_K(u)=f(u)$ on $[0,T]$. By the uniqueness of solutions, the solution of the truncated problem (\ref{K}) coincides with the solution of the original problem (\ref{semigroup}) for $t\in [0,T]$. This completes the proof of Proposition \ref{lemma-2}.
\end{proof}

\begin{remark}
In Lemma \ref{lemma-2}, the local existence time $T$ depends on $L_K$, which is the locally Lipschitz constant of $f: H^1_0(\O)\longrightarrow L^2(\O)$; nevertheless, $T$ is independent of the locally Lipschitz constant of $f: H^1_0(\O)\longrightarrow L^{\frac{m+1}{m}}(\O)$. This observation is crucial for the next step.
\end{remark}

\subsection{Completion of the Proof for the Local Existence}

To extend the existence result in Proposition \ref{lemma-2}  to the situation where the source $f(u)$ is not necessary locally Lipschitz from $H_0^1(\O)$ into $ L^2(\O)$, we employ the following truncation of the source (first used in \cite{Radu}). Namely, put:
\begin{align} \label{cutoff}
f_n(u)=f(u)\eta_n(u)
\end{align}
where $\eta_n\in C^{\infty}_0(\reals)$ is a smooth cutoff function such that: $0\leq \eta_n \leq 1$; $\eta_n(u)=1$ if $|u|\leq n$; $\eta_n(u)=0$ if $|u|\geq 2n$;
and $|\eta'(u)|\leq C/n$.

The following result is already known in \cite{BL1,RW}.
\begin{proposition} \label{prop-1}
 Suppose $m\geq 1$, $0< \e<1$. Assume $f: \reals \longrightarrow \reals $ such that  $|f'(s)|\leq C(|s|^{p-1}+1)$, where  $p\frac{m+1}{m}\leq \frac{6}{1+2\e}$. Let $f_n$ be defined in (\ref{cutoff}). Then,
\begin{itemize}
\item $f_n: H^1_0(\O)\longrightarrow L^2(\O)$ is globally Lipschitz continuous with Lipschitz constant depending on $n$.
\item
$f_n: H^{1-\e}_0(\O)\longrightarrow L^{\frac{m+1}{m}}(\O)$ is locally Lipschitz continuous with a local Lipschitz constant independent of $n$.
\end{itemize}
\end{proposition}

We are now ready to prove the local existence statement  in Theorem \ref{thm-exist}.

\begin{proof}
By using the truncated source $f_n$ defined in (\ref{cutoff}), we define the nonlinear operator
\[\mathscr A_n \left(U\right)=\begin{pmatrix}
-v \\
-\Delta u+g(v)- \mathcal L(w)-f_n(u) \\
-v+w_s
\end{pmatrix}\]
with its domain
\begin{align*}
D(\mathscr A_n)=\{&(u,v,w)\in H^1_0(\O)\times H^1_0(\O) \times H^1_{\mu}(\reals^+,H^1_0(\O)):
g(v)\in H^{-1}(\O)\cap L^1(\O), \notag\\
&-\Delta u+g(v)-\mathcal L(w)-f_n(u) \in L^2(\O), \; w(0)=0\}.
\end{align*}
By Proposition \ref{prop-1}, $f_n(u)\in L^2(\O)$ for all $u\in H^1_0(\O)$, so $D(\mathscr A_n)$ is uniform for all $n$.
For the initial data $u_0(x,t)\in L^2_{\mu}(\reals^-,H^1_0(\O))$ satisfying Assumption \ref{ass}, there exists  $u_0^n(x,t) \in H^1_{\mu}(\reals^-, C^2_0(\O))$, $n\in \naturals$, such that $u_0^n(x,t)\longrightarrow u_0(x,t)$ in
$L^2_{\mu}(\reals^-,H^1_0(\O))$ with $u_0^n(x,0)\longrightarrow u_0(x,0)$ in $H^1_0(\O)$ and $v_0^n(x,0)\longrightarrow v_0(x,0)$ in $L^2(\O)$, where $v_0^n=\frac{d}{dt}u_0^n$ and $v_0=\frac{d}{dt}u_0$.
Put $w_0^n(x,s)=u_0^n(x,0)-u_0^n(x,-s)$. Notice, $U_0^n:=(u_0^n(x,0),v_0^n(x,0),w_0^n(x,s))\in D(\mathscr A_n)$, for every $n\in \naturals$. Therefore, by Proposition \ref{lemma-2} and Proposition \ref{prop-1}, the approximate system
\begin{align} \label{approx}
U_t+\mathscr A_n U=0
\end{align}
with the initial data $U_0^n$ has a unique local strong solution
$U_n=(u_n,v_n,w_n)\in W^{1,\infty}(0,T;H)$. Thanks to  Proposition \ref{prop-1}, the life span $T$ of each solution $U_n$, given in (\ref{def-T}), is independent of $n$, since the local Lipschitz constant of the mapping $f_n: H^1(\O) \longrightarrow L^{\frac{m+1}{m}}(\O)$ is independent of $n$.
Also, we known that $T$ depends on $K$, where $K^2\geq 4(E(0)+1)$; nonetheless, since $E_n(0)\longrightarrow E(0)$, we can choose $K$ sufficiently large so that $K$ is independent of $n$. Now, by (\ref{e19}) one has $E_n(t)\leq K^2/2$, which implies the uniform boundedness of $\norm{U_n(t)}_H$ on $[0,T]$. More precisely, we have
\begin{align} \label{unibound}
\norm{U_n(t)}^2_H=\norm{\nabla u_n(t)}_2^2+\norm{u_n'(t)}_2^2+\int_0^{\infty}\norm{\nabla w_n(t,s)}_2^2 \mu(s)ds\leq K^2
\end{align}
for all $t\in [0,T]$ and all $n\in \naturals$.
By choosing $\e\leq a/2$ in (\ref{e20}) and by the fact $E_n(t)$ is uniformly  bounded on $[0,T]$, one has
\begin{align} \label{e-21}
\int_0^T \norm{u_n'}_{m+1}^{m+1} dt \leq C_K
\end{align}
for some constant $C_K>0$ depending on $K$. In addition, by Remark \ref{remark1}, one has $w_n(x,t,s)=u_n(x,t)-u_n(x,t-s)$, for all $t$, $s\geq 0$.

It follows from (\ref{unibound}) and (\ref{e-21}) that there exists
$U=(u,v,w)\in L^{\infty}(0,T;H)$ such that, on a subsequence,
\begin{align} \label{weak-conv}
U_n \longrightarrow U \text{\;\;weak}^* \text{ in \;} L^{\infty}(0,T;H)
\end{align}
and
\begin{align} \label{weak-conv-2}
u_n' \longrightarrow u' \text{\;\;weakly in\;\;} L^{m+1}(\O\times (0,T)).
\end{align}
Also, it is straightforward to show that $v=u'$ and $w(x,t,s)=u(x,t)-u(x,t-s)$
for a.e. $t,s\geq 0$.

By virtue of  (\ref{unibound}) and (\ref{weak-conv}), we infer
\begin{align} \label{bound}
E(t)=\frac{1}{2}\left(\norm{\nabla u(t)}_2^2+\norm{u'(t)}_2^2+\int_0^{\infty}\norm{\nabla w(t,s)}_2^2 \mu(s)ds\right)&=\frac{1}{2} \norm{U(t)}_H^2 \notag \\ & \leq \frac{K^2}{2}
\end{align}
for all $t\in [0,T]$. Similarly, from (\ref{e-21}) and (\ref{weak-conv-2}), it follows that
\begin{align} \label{bound-2}
\int_0^T \norm{u'}_{m+1}^{m+1} dt \leq C_K.
\end{align}

Furthermore, by using (\ref{weak-conv}) and Aubin's compactness theorem (see for instance \cite{TE}), we infer
\begin{align} \label{strong-conv}
u_n \longrightarrow u \text{\;\;strongly in\;\;} L^{\infty}(0,T;H^{1-\e}(\O)),
\end{align}
for $0<\e<1$.
Since $U_n\in D(\mathscr A_n)$ is a strong solution of (\ref{approx}), the following variational formula holds:
\begin{align} \label{var}
&(u_n'(t),\phi(t))_{\O}-(u_n'(0),\phi(0))_{\O}-\int_0^t \int_{\O}u_n'(\t)\phi'(\t)dx d\t\notag\\
&+\int_0^t \int_{\O} \nabla u_n(\t) \cdot \nabla \phi(\t) dx d\t+\int_0^t \int_{\O} g(u_n'(\t))\phi(\t) dx d\t\notag\\
&+\int_0^t \int_0^{\infty} \int_{\O} \nabla w_n(\t,s) \cdot \nabla \phi(\t) dx \mu(s) ds d\t
=\int_0^t \int_{\O} f_n(u_n(\t)) \phi(\t) dx d\t.
\end{align}
for all $\phi \in C([0,T];H^1_0(\O))\cap L^{m+1}(\O \times (0,T))$ with $\phi_t\in C([0,T];L^2(\O))$ and for a.e. $t\in [0,T]$.

Now, we fix an arbitrary $t\in [0,T]$, and show the convergence of nonlinear terms in (\ref{var}).

We shall first show:
\begin{align} \label{conv-f}
\lim_{n\longrightarrow \infty} \int_0^t \int_{\O}f_n(u_n)\phi dx d\t=\int_0^t \int_{\O}f(u)\phi dx d\t,
\end{align}
for all $\phi\in C([0,T];H^1_0(\O))\cap L^{m+1}(\O \times (0,T))$ and a.e. $t\in[0,T]$.
Indeed, we have
\begin{align} \label{e21}
&\left|\int_0^t \int_{\O}(f_n(u_n)-f(u))\phi dx d\t\right| \notag\\
&\leq \int_0^t \int_{\O}|f_n(u_n)-f_n(u)||\phi| dx d\t+
\int_0^t \int_{\O}|f_n(u)-f(u)||\phi| dx d\t.
\end{align}

We shall estimate each term on the right hand side of (\ref{e21}).  By recalling  Proposition \ref{prop-1},  we know that  $f_n$ is locally Lipschitz continuous from
$H^{1-\e}(\O)\longrightarrow L^{\frac{m+1}{m}}(\O)$ with the Lipschitz constant independent of $n$. Thus,
\begin{align} \label{e22}
&\int_0^t \int_{\O}|f_n(u_n)-f_n(u)||\phi| dx d\t  \notag\\
&\leq \left(\int_0^t \int_{\O}|f_n(u_n)-f_n(u)|^{\frac{m+1}{m}}dx d\t\right)^{\frac{m}{m+1}}
\left(\int_0^t \int_{\O}|\phi|^{m+1}dx d\t\right)^{\frac{1}{m+1}}\notag\\
&\leq C(K) \norm{\phi}_{L^{m+1}(\O \times (0,T))} \left(\int_0^t \norm{u_n-u}_{H^{1-\e}(\O)}^{\frac{m+1}{m}} d\t\right)^{\frac{m}{m+1}}\longrightarrow 0,
\end{align}
where we have used the strong convergence (\ref{strong-conv}).

The second term on the right hand side of (\ref{e21}) is handled as follows.
\begin{align} \label{e23}
&\int_0^t \int_{\O}|f_n(u)-f(u)||\phi| dx d\t\notag\\
&\leq \left(\int_0^t \int_{\O}|f_n(u)-f(u)|^{\frac{m+1}{m}} dx d\t\right)^{\frac{m}{m+1}}
\left(\int_0^t \int_{\O}|\phi|^{m+1}dx d\t\right)^{\frac{1}{m+1}}\notag\\
&\leq \norm{\phi}_{L^{m+1}(\O \times (0,T))} \left(\int_0^t \int_{\O}|f(u)|^{\frac{m+1}{m}}|\eta_n(u)-1|^{\frac{m+1}{m}} dx d\t\right)^{\frac{m}{m+1}}.
\end{align}
Thanks to  the assumptions $|f(u)|\leq C(|u|^p+1)$, $p\frac{m+1}{m}<6$,  and the Sobolev imbedding  $H^1_0(\O)\hookrightarrow L^6(\O)$, it can be easily shown that $f(u)\in L^{\frac{m+1}{m}}(\O \times (0,T))$, for each $u\in H^1_0(\O)$. Also, notice $\eta_n(u(x)) \longrightarrow 1$ a.e. in $\O$. Thus, by the Lebesgue Dominated Convergence Theorem, it follows that the right hand side of (\ref{e23}) converges to zero, and along with (\ref{e22}), we conclude the right hand side of (\ref{e21}) is also convergent to zero. Therefore,  (\ref{conv-f}) follows.

In order to deal with the term due to damping  in (\ref{var}), we recall the assumption $g(s)s\leq |s|^{m+1}$ for all $|s|\geq 1$, then by (\ref{e-21}) one has $g(u_n')$ is uniformly bounded in the space $L^{\frac{m+1}{m}}(\O \times (0,t))$. Therefore, there exists $g^*\in L^{\frac{m+1}{m}}(\O \times (0,t))$ such that,
on a subsequence
\begin{align} \label{g-conv}
g(u_n')\longrightarrow g^* \text{\;\;weakly in\;\;} L^{\frac{m+1}{m}}(\O \times (0,t)).
\end{align}
We aim to show $g^*=g(u')$. To accomplish this assertion, we consider two  solutions $U_n$ and $U_j$ of the approximate problem (\ref{approx}) corresponding to the parameters $n$ and $j$, respectively. By denoting $\tilde U=U_n-U_j$, then the following energy inequality holds:
\begin{align} \label{e24}
&\tilde E(t)
+\int_0^t \int_{\O} (g(u_n')-g(u_j'))\tilde u' dx d\t\notag\\
&\leq \tilde E(0)
+\int_0^t \int_{\O} |f_n(u_n)-f_j(u_j)| |\tilde u'|dx d\t,
\end{align}
where $\tilde E(t)=\frac{1}{2}\left(\norm{\tilde u'(t)}_2^2+\norm{\nabla \tilde u(t)}_2^2+\int_0^{\infty}\norm{\nabla \tilde w(t,s)}_2^2\mu(s)ds\right)$.

Let us show first that
 $\int_0^t \int_{\O} |f_n(u_n)-f_j(u_j)|\tilde u' dx d\t\longrightarrow 0$ as $n$, $j\longrightarrow \infty$. Indeed,
\begin{align} \label{e25}
&\int_0^t \int_{\O} |f_n(u_n)-f_j(u_j)||\tilde u'| dx d\t \notag\\
&\leq \int_0^t \int_{\O} |f_n(u_n)-f_n(u)||\tilde u'| dx d\t+\int_0^t \int_{\O} |f_n(u)-f(u)||\tilde u'| dx d\t \notag\\
&+\int_0^t \int_{\O} |f(u)-f_j(u)||\tilde u'| dx d\t+\int_0^t \int_{\O} |f_j(u)-f_j(u_j)||\tilde u'| dx d\t.
\end{align}
By replacing $\phi$ by $\tilde u'$ in (\ref{e22}) and (\ref{e23}), we conclude  that the right hand side of (\ref{e25}) converges to zero, and along with the fact $U^n_0\longrightarrow U_0$ in $H$, it follows that the right hand side of (\ref{e24}) converges to zero. Thus,
\begin{align} \label{e26}
\int_0^t (g(u_n')-g(u_j'))(u_n'-u_j') dx d\t \longrightarrow 0
\end{align}
as $n$, $j\longrightarrow \infty$. Note, since the function $g$ is increasing, it is straightforward to show the operator $g(\cdot): L^{m+1}(\O \times (0,t))\longrightarrow L^{\frac{m+1}{m}}(\O \times (0,t))$ is monotone and hemi-continuous, which implies it is maximal monotone. Thus, with this in hand, it follows from (\ref{g-conv}) and (\ref{e26}) that (see Lemma 2.3 in \cite{Barbu3}), $g^*=g(u')$. Hence,
\begin{align} \label{conv-g}
g(u_n')\longrightarrow g(u') \text{\;\;weakly in\;\;} L^{\frac{m+1}{m}}(\O \times (0,t)).
\end{align}

Now, by (\ref{weak-conv}), (\ref{conv-f}) and (\ref{conv-g}), we can pass to limit on (\ref{var}), and conclude that (\ref{weak}) holds.

It remains to show the continuity of the solution, and verify that $u$ satisfies   the initial condition. Indeed, since $g$ is a increasing function, then (\ref{e24}) yields
\begin{align} \label{e27}
\tilde E(t)\leq \tilde E(0)+\int_0^T \int_{\O} |f_n(u_n)-f_j(u_j)| |\tilde u'|dx d\t,
\text{\;\;\;for all\;\;\;} t\in [0,T].
\end{align}
Notice that  the right hand side of (\ref{e27}) does not depend on the value of $t$ and it does converge to zero as $n$, $j\longrightarrow \infty$. Therefore, it follows that
\begin{align} \label{e28}
u_n \longrightarrow u \text{\;\;in\;\;} H^1_0(\O) \text{\;\;and\;\;} u'_n \longrightarrow u' \text{\;\;in\;\;} L^2(\O),  \text{\;\;uniformly on\;\;} [0,T].
\end{align}
Since $u_n\in W^{1,\infty}([0,T];H^1_0(\O))$ and $u_n'\in W^{1,\infty}([0,T];L^2(\O))$, then the uniform convergence in (\ref{e28})  implies that
$u\in C([0,T];H^1_0(\O))$ and $u'\in C([0,T];L^2(\O))$. In addition, (\ref{e28}) shows
$u_n(x,0)\longrightarrow u(x,0)$ in $H^1_0(\O)$, and since  $u^n_0(x,0)\longrightarrow u_0(x,0)$ in $H^1_0(\O)$, then
 $u(x,0)=u_0(x,0)\in H^1(\O)$. Also, since $u(x,t)=u_0(x,t)$ for $t<0$, it follows that $u(x,t)=u_0(x,t)$ for all $t\leq 0$. This completes the proof of the local  existence statement in Theorem \ref{thm-exist}.
\end{proof}

\begin{remark}\label{rem-2.6}
It can be easily shown, the weak solutions obtained in the proof above, satisfy the following energy \textbf{inequality}:
\begin{align}\label{enrgy-ineq}
E(t)+\int_0^t \int_{\O}g(u_t)u_t dx d\tau\leq E(0)+\int_0^t \int_{\O} f(u)u_t dx d\tau.
\end{align}
In the next section, we shall prove all weak solutions of (\ref{1.1}) in the sense of Definition \ref{def-weak}, satisfy the {\bf energy identity} (\ref{EI-0}) and we will use this fact to justify the uniqueness of solutions.
\end{remark}

\bigskip

\section{Uniqueness of weak solutions}
\subsection{Energy Identity}

In order to prove the uniqueness of weak solutions, we shall justify the energy identity (\ref{EI-0}) rigorously. Notice that the energy identity can be derived formally by testing the equation (\ref{1.1}) by $u_t$, however, such calculation is not rigorous, due to the fact that $u_t$ is not sufficiently regular to be the test function in as required in Definition \ref{def-weak}. To resolve this issue, we employ the operator $T_{\epsilon}:=(I-\epsilon \Delta)^{-1}$ to smooth functions in space.
Some important properties of $T_{\epsilon}$ can be found in the Appendix.
Denote $u_{\epsilon}=T_{\epsilon}u$.

We prove the energy identity (\ref{EI-0}) as follows.
\begin{proof}
Act the regularizing operator $T_{\epsilon}$ on every term of the equation and multiply by $u'_{\epsilon}$. After integrating in space and time, we obtain
\begin{align}  \label{Gal-1}
&\int_0^t \int_{\O}  u''_{\epsilon}(\t) \cdot  u_{\epsilon}'(\t) dx d\t +\int_0^t \int_{\O} \nabla u_{\epsilon}(\t) \cdot \nabla u_{\epsilon}'(\t) dx d\t \notag\\
&+\int_0^t \int_0^{\infty} \int_{\O} \nabla (u_{\epsilon}(\t)-u_{\epsilon}(\t-s)) \cdot \nabla u'_{\epsilon}(\t) dx \mu(s) ds d\t \notag\\
&+\int_0^t \int_{\O} T_{\epsilon}(g(u'(\t))) u_{\epsilon}'(\t) dx d\t
=\int_0^t \int_{\O} T_{\epsilon} (f(u(\t))) u_{\epsilon}'(\t) dx d\t.
\end{align}

Since $u\in H^1_0(\O)$ and $u' \in L^2(\Omega)$, by means of Proposition \ref{H1conv}, we have $u_{\epsilon}\rightarrow u$ in $H^1_0(\O)$ and $u'_{\epsilon}\rightarrow u'$ in $L^2(\O)$.
Thus, the first two integrals on the left-hand side of (\ref{Gal-1}) converges to $$\frac{1}{2}(\|u'(t)\|_{2}^2-\|u'(0)\|_2^2+\|\nabla u(t)\|_{2}^2-\|\nabla u(0)\|^2_{2}).$$

Recall $u'\in L^{m+1}(\O)$, $g(u')\in L^{\frac{m+1}{m}}(\O)$, $m\geq 1$. By Proposition \ref{norm-conveg}, one has $\|u'_{\epsilon}\|_{m+1}\leq \|u'\|_{m+1}$ and  $u'_{\epsilon}\rightarrow u'$ in $L^{m+1}(\Omega)$, as well as
$\|T_{\epsilon}(g(u'))\|_{\frac{m+1}{m}}\leq \|g(u')\|_{\frac{m+1}{m}}$ and
 $T_{\epsilon}(g(u'))\rightarrow g(u')$ in $L^{\frac{m+1}{m}}(\Omega)$. Consequently, we can pass to the limit for the damping term in (\ref{Gal-1}) by using the Lebesgue Dominated Convergence Theorem. This shows
 \begin{align*}
 \lim_{\epsilon\rightarrow 0} \int_0^t \int_{\O} T_{\epsilon}(g(u'(\t))) u_{\epsilon}'(\t) dx d\t=\int_0^t \int_{\O} g(u'(\t)) u'(\t) dx d\t.
 \end{align*}
 On the other hand, $u\in H^1(\O)$ implies $f(u)\in L^{6/p}(\O)$ due to the growth rate of $f$ is $p$. With the assumption $p\frac{m+1}{m}<6$, we have $f(u)\in L^{\frac{m+1}{m}}(\O)$, and therefore similar as the damping term, we can prove the convergence of the source term in (\ref{Gal-1}). That is
 \begin{align*}
 \lim_{\epsilon\rightarrow 0}\int_0^t \int_{\O} T_{\epsilon} (f(u(\t))) u_{\epsilon}'(\t) dx d\t=\int_0^t \int_{\O} f(u(\t)) u'(\t) dx d\t.
 \end{align*}

It remains to treat the memory term in (\ref{Gal-1}). We split it as
\begin{align} \label{Gal-2}
&\int_0^t \int_0^{\infty} \int_{\O} \nabla (u_{\epsilon}(\t)-u_{\epsilon}(\t-s)) \cdot \nabla u'_{\epsilon}(\t) dx \mu(s) ds d\t \notag\\
&=\int_0^t \int_0^{\infty} \int_{\O} \nabla (u_{\epsilon}(\t)-u_{\epsilon}(\t-s)) \cdot \nabla (u'_{\epsilon}(\t)-u'_{\epsilon}(\t-s)) dx \mu(s) ds d\t \notag\\
&\hspace{0.2 in}+\int_0^t \int_0^{\infty} \int_{\O} \nabla (u_{\epsilon}(\t)-u_{\epsilon}(\t-s)) \cdot \nabla u'_{\epsilon}(\t-s) dx \mu(s) ds d\t.
\end{align}
Notice, by virtue of Proposition \ref{H1conv}, the first integral on the right-hand side of (\ref{Gal-2}) converges to
\begin{align*}
\frac{1}{2}\int_0^{\infty} \left(\norm{\nabla(u(t)-u(t-s))}_2^2-\norm{\nabla(u(0)-u(-s))}_2^2 \right) \mu(s) ds,
\end{align*}
where we have used the assumption $u_0(x,t)\in L^2_{\mu}(\mathbb R^-, H^1_0(\Omega))$.
Next, we consider the second term on the right-hand side of (\ref{Gal-2}):
\begin{align} \label{Gal-4}
 &\int_0^t \int_0^{\infty} \int_{\O} \nabla (u_{\epsilon}(\t)-u_{\epsilon}(\t-s)) \cdot \nabla \partial_{\t} u_{\epsilon}(\t-s) dx \mu(s) ds d\t \notag\\
 &= \int_0^t \int_0^{\infty} \int_{\O} \nabla (u_{\epsilon}(\t)-u_{\epsilon}(\t-s)) \cdot \nabla \partial_{s} (u_{\epsilon}(\tau)-u_{\epsilon}(\t-s)) dx \mu(s) ds d\t \notag\\
 &=\frac{1}{2}\int_0^t \int_0^{\infty} \|\nabla(u_{\epsilon}(\t)-u_{\epsilon}(\t-s))\|_2^2 (-\mu'(s)) ds d\tau,
 \end{align}
 where we performed the integration by parts with respect to the time variable $s$ and used the assumption $\mu(\infty)=0$. Notice that (\ref{Gal-4}) is non-negative due to $\mu'(s)\leq 0$.
 In order to see that the limit of (\ref{Gal-4}) exists as $\e\rightarrow 0$, we employ (\ref{Gal-1})- (\ref{Gal-4} to get:
 \begin{align} \label{Gal-5}
 0&\leq \lim_{\epsilon\rightarrow 0}\frac{1}{2}\int_0^t \int_0^{\infty} \|\nabla(u_{\epsilon}(\t)-u_{\epsilon}(\t-s))\|_2^2 (-\mu'(s)) ds d\tau  \notag\\
 &=-\lim_{\epsilon\rightarrow 0}\int_0^t \int_0^{\infty} \int_{\O} \nabla (u_{\epsilon}(\t)-u_{\epsilon}(\t-s)) \cdot \nabla (u'_{\epsilon}(\t)-u'_{\epsilon}(\t-s)) dx \mu(s) ds d\t \notag\\
 &-\lim_{\epsilon\rightarrow 0}\int_0^t \int_{\O}  u''_{\epsilon}(\t) \cdot  u_{\epsilon}'(\t) dx d\t -\lim_{\epsilon\rightarrow 0}\int_0^t \int_{\O} \nabla u_{\epsilon}(\t) \cdot \nabla u_{\epsilon}'(\t) dx d\t  \notag\\
 &-\lim_{\epsilon\rightarrow 0}\int_0^t \int_{\O} T_{\epsilon}(g(u'(\t))) u_{\epsilon}'(\t) dx d\t+\lim_{\epsilon\rightarrow 0}\int_0^t \int_{\O} T_{\epsilon} (f(u(\t))) u_{\epsilon}'(\t) dx d\t<\infty,
 \end{align}
since we have shown that each limit on the right-hand side of (\ref{Gal-5}) converges to a finite value.

Applying the convergence of $u_{\epsilon}$ in $H^1_0(\O)$ and (\ref{Gal-5}), it follows from Fatou's Lemma that
 \begin{align} \label{Gal-3}
&\frac{1}{2}\int_0^t \int_0^{\infty} \|\nabla(u(\t)-u(\t-s))\|_{2}^2 (-\mu'(s)) ds d\tau \notag\\
&\leq \lim_{\epsilon\rightarrow 0}\frac{1}{2}\int_0^t \int_0^{\infty} \|\nabla(u_{\epsilon}(\t)-u_{\epsilon}(\t-s))\|_{2}^2 (-\mu'(s)) ds d\tau \notag\\
&\leq \frac{1}{2}\int_0^t \int_0^{\infty} \|\nabla(u(\t)-u(\t-s))\|_{2}^2 (-\mu'(s)) ds d\tau,
\end{align}
where the last inequality is due to the fact $\|\nabla(u_{\epsilon}(\t)-u_{\epsilon}(\t-s))\|_2 \leq \|\nabla(u(\t)-u(\t-s))\|_2$
by means of Proposition \ref{H1conv}. This implies that the inequalities in (\ref{Gal-3}) are actually equalities, which gives us the desired limit of (\ref{Gal-4}) as $\e\rightarrow 0$.
\end{proof}

\smallskip

\subsection{Continuous Dependence on Initial Data}
This subsection is devoted to the proof of  Theorem \ref{thm-cont}, which states that the solution of system (\ref{1.1}) depends  continuously on the initial data. The uniqueness of solutions then follows immediately.
\begin{proof}
Let $u^n_0$, $u_0\in L^2_{\mu}(\reals^-,H^1_0(\O))$, $n\in \naturals$, such that
$u_0^n\longrightarrow u_0$ in $L^2_{\mu}(\reals^-,H^1_0(\O))$ with
$u^n_0(0)\longrightarrow u_0(0)$ in $H^1_0(\O)$ and in $L^{\frac{3(p-1)}{2}}(\O)$, $\frac{d}{dt}u^n_0(0)\longrightarrow \frac{d}{dt}u_0(0)$ in $L^2(\O)$.
Let $\{u_n\}$ and $u$ be weak solutions on $[0,T]$ corresponding to the initial data
$\{u_0^n\}$ and $u_0$, respectively. It is important to note here that  the local existence time $T$ can be selected independent of $n$. To see this, recall (\ref{def-T}) and (\ref{bound-K}),  which imply that  $T$ depends on $K$, while $K$ depends on the initial energy.  Nonetheless, we can choose $K$ sufficiently large so that $K^2\geq 4(E_n(0)+1)$, for all $n$, and  $K^2\geq 4(E(0)+1)$, where $E_n(t)$ and $E(t)$ are quadratic energies corresponding to $u_n$ and $u$, respectively. Thus, $K$ and $T$ are both independent of $n$. Furthermore,  (\ref{bound}) and (\ref{bound-2}) yield
\begin{align} \label{bound-3}
\begin{cases}
E_n(t)\leq K^2/2, \quad E(t) \leq K^2/2, \text{\;\;for all\;\;} t \in [0,T];\\[.1in]
\ds \int_0^T \norm{u'_n}_{m+1}^{m+1}dt\leq C_K,\;\;\int_0^T \norm{u'}_{m+1}^{m+1}dt\leq C_K,
\end{cases}
\end{align}
for all $n\in \naturals$.

Now put:  $\tilde u_n=u_n-u$ and
\begin{align*}
\tilde E_n(t):=\frac{1}{2}\left(\norm{\tilde u_n'(t)}_2^2+\norm{\nabla \tilde u_n(t)}_2^2+\int_0^{\infty}\norm{\nabla \tilde w_n(t,s)}_2^2 \mu(s)ds \right),
\end{align*}
where $\tilde w_n(x,t,s)=\tilde u_n(x,t)-\tilde u_n(x,t-s)$.
By assumption, $\tilde E_n(0)\longrightarrow 0$. We aim to show $\tilde E_n(t) \longrightarrow 0$ uniformly on $[0,T]$.

Following the same approach in proving the energy identity (\ref{EI-0}), one can obtain that:
\begin{align*}
&\tilde  E_n(t)+\int_0^t \int_{\O}(g(u_n')-g(u'))(u_n'-u') dx d\t-\frac{1}{2}\int_0^t \int_0^{\infty} \norm{\nabla \tilde w_n}_2^2 \mu'(s)ds d\t \notag\\
&=\tilde E_n(0)+\int_0^t \int_{\O} (f(u_n)-f(u))(u_n'-u') dx d\t.
\end{align*}

By the monotonicity of the function $g$ and the assumption that $\mu'<0$, then the following  energy inequality holds:
\begin{align} \label{u-1}
\tilde  E_n(t) \leq \tilde E_n(0)+\int_0^t \int_{\O} (f(u_n)-f(u))\tilde u_n' dx d\t.
\end{align}

If $1\leq p\leq 3$, then the uniqueness of weak solutions can be obtained immediately by the energy inequality (\ref{u-1}). To see this, we recall, if $1\leq p\leq 3$, then $f: H^1_0(\O)\longrightarrow L^2(\O)$ is locally Lipschitz continuous, and along with (\ref{bound-3}), we infer
\begin{align*}
&\int_0^t \int_{\O}(f(u_n)-f(u))\tilde u_n' dx d\t  \notag\\
&\leq C(K)\left(\int_0^t \norm{\nabla \tilde u_n}_2^2 d\t\right)^{\frac{1}{2}}
\left(\int_0^t \norm{\tilde u_n'}_2^2 d\t\right)^{\frac{1}{2}}\leq C(K)\int_0^t \tilde E_n(\t)d\t.
\end{align*}
Thus, it follows from (\ref{u-1}) that
\begin{align*}
\tilde E_n(t)\leq \tilde E_n(0)+C(K)\int_0^t \tilde E_n(\t)d\t, \text{\;\;\;for all\;\;\;} t\in [0,T].
\end{align*}
By Gronwall's inequality, we have $$\tilde E(t)\leq C(K,T) \tilde E_n(0)$$ for all $t \in [0,T]$.  Since $\tilde E_n(0)\longrightarrow 0$, then $\tilde E_n(t)\longrightarrow 0$ uniformly on $[0,T]$, for the case $1\leq p\leq 3$.

However, if $3<p<6$, the estimate for the source term is  more subtle. Here, we follow a clever idea that has been used in \cite{BL2}. As in \cite{BL2}, we shall perform integration by parts \emph{twice} with respect to the time variable $t$, which essentially   convert $\tilde u_n'\in L^2(\O)$ to the more regular term $\tilde u_n\in H_0^1(\O)\hookrightarrow L^6(\O)$. More precisely, we compute as follows:
\begin{align} \label{u-2}
&\int_0^t \int_{\O} (f(u_n)-f(u))\tilde u_n' dx d\t \notag\\
&=\left[\int_{\O}(f(u_n)-f(u))\tilde u_n dx \right]_0^t
-\int_0^t \int_{\O}(f_1'(u_n)u_n'-f'(u)u')\tilde u_n dx d\t \notag\\
&=\left[\int_{\O}(f(u_n)-f(u))\tilde u_n dx \right]_0^t
-\int_0^t \int_{\O} (f'(u_n)-f'(u))u_n' \tilde u_n dx d\t \notag\\
&\hspace{1 in}-\int_0^t \int_{\O} f'(u)\tilde u_n' \tilde u_n dx d\t.\notag\\
&=\left[\int_{\O}(f(u_n)-f(u))\tilde u_n dx \right]_0^t
-\int_0^t \int_{\O} (f'(u_n)-f'(u))u_n' \tilde u_n dx d\t \notag\\
&\hspace{1 in}-\frac{1}{2}\left[\int_{\O}f'(u) |\tilde u_n|^2 dx\right]_0^t
+\frac{1}{2}\int_0^t \int_{\O} f''(u)u'|\tilde u_n|^2 dx d\t.
\end{align}
By the assumptions on $f$, we have
\begin{align} \label{u-3}
\begin{cases}
|f''(u)|\leq C(|u|^{p-2}+1)\\
|f'(u)|\leq C(|u|^{p-1}+1)\\
|f(u_n)-f(u)|\leq C(|u_n|^{p-1}+|u|^{p-1}+1)|\tilde u_n|\\
|f'(u_n)-f'(u)|\leq C(|u_n|^{p-2}+|u|^{p-2}+1)|\tilde u_n|.
\end{cases}
\end{align}
By using  (\ref{u-3}), then we estimate (\ref{u-2}) as follows:
\begin{align}  \label{u-4}
&\int_0^t \int_{\O} (f(u_n)-f(u))\tilde u_n' dx d\t \notag\\
&\leq \int_{\O}(|\tilde u_n(t)|^2+|\tilde u_n(0)|^2) dx
+\int_{\O}(|u_n(t)|^{p-1}+|u(t)|^{p-1})|\tilde u_n(t)|^2 dx \notag\\
&+\int_{\O}(|u_n(0)|^{p-1}+|u(0)|^{p-1})|\tilde u_n(0)|^2 dx
+\int_0^t \int_{\O}|\tilde u_n|^2 (|u_n'|+|u'|)dx d\t \notag\\
&+\int_0^t \int_{\O} (|u_n|^{p-2}+|u|^{p-2})|\tilde u_n|^2(|u_n'|+|u'|) dx d\t \notag\\
&:=I_1+I_2+I_3+I_4+I_5.
\end{align}

The next step is to estimate each term on the right hand side of (\ref{u-4}). First, let us look at
\begin{align} \label{u-5}
&I_1=\int_{\O}(|\tilde u_n(t)|^2+|\tilde u_n(0)|^2) dx
=\int_{\O}\left(\left|\tilde u_n(0)+\int_0^t \tilde u_n'(\t)d\t\right|^2+|\tilde u_n(0)|^2 \right) dx \notag\\
&\leq 3 \norm{\tilde u_n(0)}_2^2 + 2t \int_0^t \norm{\tilde u_n'(\t)}^2_2 d\t
\leq C \left(\tilde E_n(0)+ T \int_0^t \tilde E_n(\t) d\t\right).
\end{align}
Also, by H\"older's inequality and the imbedding $H^1(\O)\hookrightarrow L^6(\O)$, one has
\begin{align} \label{u-6}
I_3&=\int_{\O}(|u_n(0)|^{p-1}+|u(0)|^{p-1})|\tilde u_n(0)|^2 dx  \notag\\
& \leq \left(\norm{u_n(0)}^{p-1}_{\frac{3(p-1)}{2}}+\norm{u(0)}^{p-1}_{\frac{3(p-1)}{2}}\right)
\norm{\tilde u_n(0)}_6^2
\leq C\tilde E_n(0),
\end{align}
where we have used the fact $u_n(0)=u^n_0(0)\longrightarrow u(0)=u_0(0)$ in $L^{\frac{3(p-1)}{2}}(\O)$.

Similarly,
\begin{align} \label{u-7}
I_4&=\int_0^t \int_{\O}|\tilde u_n|^2 (|u_n'|+|u'|)dx d\t
\leq C\int_0^t \norm{\tilde u_n}_6^2(\norm{u_n'}_2+\norm{u'}_2)d\t  \notag\\
&\leq C\int_0^t \norm{\nabla \tilde u_n}_2^2(\norm{u_n'}_2+\norm{u'}_2)d\t
\leq C(K)\int_0^t \tilde E_n(\t) d\t.
\end{align}

To estimate $I_5$, we recall the assumption $p\frac{m+1}{m}<6$, which implies $\frac{6}{6-p}<m+1$. Hence,
\begin{align} \label{u-8}
I_5&=\int_0^t \int_{\O} (|u_n|^{p-2}+|u|^{p-2})|\tilde u_n|^2(|u_n'|+|u'|) dx d\t \notag\\
&\leq C\int_0^t \left(\norm{u_n}_6^{p-2}+\norm{u}_6^{p-2}\right)\norm{\tilde u_n}_6^2
\left(\norm{u_n'}_{\frac{6}{6-p}}+\norm{u'}_{\frac{6}{6-p}}\right)d\t \notag\\
&\leq C(K) \int_0^t \tilde E_n(\t) \left(\norm{u_n'}_{m+1}+\norm{u'}_{m+1}\right) d\t.
\end{align}

Finally, we estimate $I_2=\int_{\O}(|u_n(t)|^{p-1}+|u(t)|^{p-1})|\tilde u_n(t)|^2 dx$. For the sake of clarification, we focus on the term $\int_{\O} |u_n(t)|^{p-1} |\tilde u_n(t)|^2 dx$.
The estimate for $\int_{\O} |u(t)|^{p-1} |\tilde u_n(t)|^2 dx$ will be the same. There are two different cases to be considered.

{\bf \emph{Case 1: $3<p<5$.}} In this case, we split the integral to obtain
\begin{align} \label{u-9}
\int_{\O} |u_n(t)|^{p-1} |\tilde u_n(t)|^2 dx
\leq \int_{\O}|\tilde u_n(t)|^2 dx+\int_{\{x\in \O:\,\,  |u_n(t)|>1\}}|u_n(t)|^{p-1}|\tilde u_n(t)|^2 dx.
\end{align}
Note that the first term on the right hand side of (\ref{u-9}) has been estimated in (\ref{u-5}). So,  we only consider the second term. Let $\e\in (0,5-p)$, so if $|u_n|>1$,
then $|u_n|^{p-1}<|u_n|^{4-\e}$. It follows that
\begin{align} \label{u-10}
&\int_{\{x\in \O:|u(t)|>1\}}|u_n(t)|^{p-1}|\tilde u_n(t)|^2 dx
\leq \int_{\O}|u_n(t)|^{4-\e}|\tilde u_n(t)|^2 dx
\leq \norm{u_n(t)}_6^{4-\e}\norm{\tilde u_n(t)}^2_{\frac{6}{1+\e/2}}\notag\\
&\leq C\norm{\nabla u_n(t)}_2^{4-\e}\norm{\tilde u_n(t)}^2_{H^{1-\e/4}(\O)}
=C(K)(\e\norm{\nabla \tilde u_n(t)}_2^2+C_{\e}\norm{\tilde u_n(t)}_2^2)
\end{align}
where we have use the imbedding $H^{1-\delta}(\O)\hookrightarrow L^{\frac{6}{1+2\delta}}(\O)$ and the interpolation inequality. We infer from (\ref{u-5}), (\ref{u-9}) and (\ref{u-10}) that
\begin{align} \label{u-11}
\int_{\O} |u_n(t)|^{p-1} |\tilde u_n(t)|^2 dx
\leq C(K,\e)\left(\tilde E_n(0)+T \int_0^t \tilde E_n(\t)d\t\right)
+C(K)\e\tilde E_n(t).
\end{align}

{\bf \emph{Case 2: $5\leq p<6$.}} In this case, we require the initial data
$u^n_0(0)$, $u_0(0)\in L^{\frac{3(p-1)}{2}}(\O)$. Note, for any $\e>0$, there exists $\phi\in C_0(\O)$ such that
$\norm{u_0(0)-\phi}_{\frac{3(p-1)}{2}}<\e^{\frac{1}{p-1}}$. We consider
\begin{align} \label{u-12}
&\int_{\O}|u_n(t)|^{p-1}|\tilde u_n(t)|^2 dx  \notag\\
&\leq C\int_{\O}|u_n(t)-u^n_0(0)|^{p-1}|\tilde u_n(t)|^2 dx+C\int_{\O}|u^n_0(0)-u_0(0)|^{p-1}|\tilde u_n(t)|^2 dx \notag\\
&\hspace{0.5 in}+C\int_{\O}|u_0(0)-\phi|^{p-1}|\tilde u_n(t)|^2 dx
+C\int_{\O}|\phi|^{p-1}|\tilde u_n(t)|^2 dx.
\end{align}

By the assumption $p\frac{m+1}{m}<6$ and $5\leq p<6$, we infer $m>5$. In addition, we have $$\frac{3(p-1)}{2(m+1)}< \frac{3(5m-1)}{2(m+1)^2}< 1, $$
for $m>5$. Therefore, the first term on the right hand side of (\ref{u-12}) can be estimated as follows.
\begin{align} \label{u-13}
\int_{\O}|u_n(t)-u^n_0(0)|^{p-1}|\tilde u_n(t)|^2
&\leq \left(\int_{\O}|u_n(t)-u_n(0)|^{\frac{3(p-1)}{2}}\right)^{2/3}\norm{\tilde u_n(t)}_6^2  \notag\\
&\leq C\left(\int_{\O}\left| \int_0^t u_n'(\t) d\t \right|^{\frac{3(p-1)}{2}} dx \right)^{2/3} \norm{\tilde u_n(t)}^2_{1,\O} \notag\\
&\leq C\left[\int_{\O} \left(\int_0^t |u_n'|^{m+1} d\t \right)^{\frac{3(p-1)}{2(m+1)}} dx\right]^{2/3} T^{\frac{m(p-1)}{m+1}} \tilde E_n(t) \notag\\
& \leq C(K)T^{\frac{m(p-1)}{m+1}} \tilde E_n(t),
\end{align}
where we have used the bound $\int_0^T \norm{u_n'}_{m+1}^{m+1} dt\leq K$ and the fact that  $\frac{3(p-1)}{2(m+1)}<1$.

Next, we consider the second term on the right hand side of (\ref{u-12}).
\begin{align} \label{u-14'}
\int_{\O} &|u^n_0(0)-u_0(0)|^{p-1}|\tilde u_n(t)|^2dx
\leq \norm{u^n_0(0)-u_0(0)}_{\frac{3(p-1)}{2}}^{p-1} \norm{\tilde u_n(t)}_6^2 \leq \e \tilde E_n(t),
\end{align}
for $n$ sufficiently large, due to the assumption $u^n_0(0)\longrightarrow u_0(0)$ in $L^{\frac{3(p-1)}{2}}(\O)$.

Similarly, we have
\begin{align} \label{u-14}
\int_{\O} &|u_0(0)-\phi|^{p-1}|\tilde u_n(t)|^2dx
\leq \norm{u_0(0)-\phi}_{\frac{3(p-1)}{2}}^{p-1} \norm{\tilde u_n(t)}_6^2 \leq C\e \tilde E_n(t).
\end{align}

In addition, since $\phi \in C_0(\O)$, it is clear that $|\phi(x)|\leq C(\e)$, for all $x\in \O$. So, by (\ref{u-5}), we estimate the last term on the right hand side of (\ref{u-12}) as follows:
\begin{align} \label{u-15}
\int_{\O} |\phi|^{p-1}|\tilde u_n(t)|^2dx
\leq C(\e) \int_{\O} |\tilde u_n(t)|^2 dx
\leq C(\e)\left(\tilde E_n(0)+T\int_0^t \tilde E_n(\t) d\t\right).
\end{align}

It follows from (\ref{u-12})-(\ref{u-15}) that
\begin{align} \label{u-16}
&\int_{\O} |u_n(t)|^{p-1}|\tilde u_n(t)|^2 dx \notag\\
&\leq C(K)\left(T^{\frac{m(p-1)}{m+1}}+\e \right)\tilde E_n(t)
+C(\e)\left(\tilde E_n(0)+T\int_0^t \tilde E_n(\t)d\t \right),
\end{align}
in the case $5\leq p<6$.

By combining   (\ref{u-11})
and (\ref{u-16}) for the both cases,  we conclude that,
\begin{align} \label{u-17}
I_2 \leq C(K)\left(T^{\frac{m(p-1)}{m+1}}+\e \right)\tilde E_n(t)
+C(K,\e)\left(\tilde E_n(0)+T\int_0^t \tilde E_n(\t)d\t \right),
\end{align}
for any $3<p<6$.

Now, by  combining  (\ref{u-1}), (\ref{u-4})-(\ref{u-8}) and (\ref{u-17}), we have
\begin{align*}
\tilde E_n(t)
\leq  &C(K)\left(T^{\frac{m(p-1)}{m+1}}+\e \right)\tilde E_n(t)
+C(K,\e)\tilde E_n(0) \notag\\
&+C(K,T,\e)\int_0^t \tilde E_n(\t)\left( \norm{u_n'}_{m+1}+\norm{u'}_{m+1}+1 \right)  d\t,
\end{align*}
for all $t\in [0,T]$. By selecting $\e$ and $T$ sufficiently small so that $$C(K)\left(T^{\frac{m(p-1)}{m+1}}+\e \right) <1,$$
then by  Gronwall's inequality,
\begin{align*}
\tilde E_n(t)\leq C(K,T,\e)\tilde E_n(0)\exp \left(\int_0^t \left( \norm{u_n'}_{m+1}+\norm{u'}_{m+1}+1 \right)  d\t\right).
\end{align*}
Hence,
\begin{align*}
\tilde E_n(t)\leq C(K,T,\e)\tilde E_n(0),
\end{align*}
and since $\tilde E_n(0)\longrightarrow 0$, we conclude that $\tilde E_n(t)\longrightarrow 0$ uniformly on $[0,T]$.
\end{proof}

\bigskip

\section{Global existence}
In this section we prove Theorem \ref{thm-global} stating that a local weak solution $u$ on $[0,T]$ can be extended to $[0,\infty)$ provided the damping term $g(u_t)$ dominates the source $f(u)$, i.e., $m\geq p$.

\begin{proof}
One may employ a standard continuation argument from ODE theory to obtain that the weak solution $u$ is either global or there exists $0<T_{max}<\infty$ such that
\begin{align} \label{contr}
\limsup_{t\longrightarrow T_{max}^-}\mathcal E(t)=+\infty
\end{align}
where $\mathcal E(t)$ is the modified energy defined by
\begin{align}  \label{energy-modi}
\mathcal E(t)=E(t)+\frac{1}{p+1}\norm{u(t)}_{p+1}^{p+1}
\end{align}
and the quadratic energy $E(t)$ is as defined in (\ref{energy}). We aim to show the latter cannot happen if $m\geq p$.

Let $u$ be a weak solution to (\ref{1.1}) on $[0,T]$ in the sense of Definition \ref{def-weak}. With the modified energy $\mathcal E(t)$, the energy identity (\ref{EI-0}) now reads
\begin{align} \label{glob-1}
\mathcal E(t)& +\int_0^t \int_{\O} g(u_t)u_t dx d\t-\frac{1}{2}\int_0^t \int_0^{\infty} \norm{\nabla w}_2^2 \mu'(s) ds d\t \notag\\
&=\mathcal E(0)+\int_0^t \int_{\O} f(u)u_t dx d\t+\int_0^t \int_{\O}|u|^{p-1}u u_t dx d\t.
\end{align}
By the assumption $|f(u)|\leq C(|u|^p+1)$, we deduce
\begin{align} \label{glob-2}
&\left|\int_0^t \int_{\O}f(u)u_t dx d\t\right|\leq C\int_0^t \int_{\O} (|u|^p+1)|u_t| dx d\t \notag\\
&\leq \int_0^t \norm{u_t}_{p+1}\left(\norm{u}_{p+1}^p+|\O|^{\frac{p}{p+1}}\right) d\t \notag\\
&\leq \e\int_0^t \norm{u_t}_{p+1}^{p+1} d\t + C_{\e}\int_0^t \left(\norm{u}_{p+1}^{p+1}+|\O|\right)d\t \notag\\
&\leq \e \int_0^t \norm{u_t}_{p+1}^{p+1} d\t+C_{\e}\int_0^t \mathcal E(\t)d\t+C_{\e}T|\O|,
\end{align}
where $\e>0$ to be chosen later.
Similarly, we have
\begin{align} \label{glob-3}
\left|\int_0^t \int_{\O} |u|^{p-1} u u_t dx d\t\right|
\leq \e\int_0^t \norm{u_t}_{p+1}^{p+1} d\t +C_{\e} \int_0^t \mathcal E(\t) d\t.
\end{align}

It follows from (\ref{glob-2})-(\ref{glob-3})  and the assumptions $g(s)s\geq a|s|^{m+1}$, $\mu(s)\leq 0$, we infer from (\ref{glob-1}) that
\begin{align} \label{glob-4}
\mathcal E(t)+a\int_0^t \norm{u_t}_{m+1}^{m+1} d\t
&\leq \mathcal E(0)+\e\int_0^t \norm{u_t}_{p+1}^{p+1} d\t +C_{\e} \int_0^t \mathcal E(\t) d\t+C_{T,\e}\notag\\
&\leq \mathcal E(0)+\e\int_0^t \norm{u_t}_{m+1}^{m+1} d\t +C_{\e} \int_0^t \mathcal E(\t) d\t+C_{T,\e},
\end{align}
where we have used the assumption $m\geq p$,  H\"older's and Young's inequalities.

Now, if we choose $\e<a$, then (\ref{glob-4}) yields
\begin{align*}
\mathcal E(t)\leq \mathcal E(0)+C_{\e}\int_0^t \mathcal E(\t) d\t+ C_{T,\e}.
\end{align*}
By Gronwall's inequality, we conclude
\begin{align*}
\mathcal E(t)\leq (\mathcal E(0)+C_{T,\e})e^{C_{\e}T},
\end{align*}
for all $t\in [0,T]$. Thus, (\ref{contr}) cannot happen, which implies that $u$ is a global weak solution.
\end{proof}
\smallskip

\section{Appendix}
Here we provide some properties of the regularizing operator $(I-\epsilon \Delta)^{-1}$ that we used in the justification of the energy identity.
Let $u\in L^p(\O)$, $1<p<\infty$. Set $u_{\epsilon}:=(I-\epsilon \Delta)^{-1}u$ with $u_{\epsilon}=0$ on $\partial \Omega$.
\begin{proposition} \label{H1conv}
The following statements hold.
\begin{itemize}
\item If $u\in L^2(\O)$, then $\|u_{\epsilon}\|_2\leq \|u\|_2$, and $u_{\epsilon}\rightarrow u$ in $L^2(\Omega)$, as $\e \rightarrow 0$.
\item If $u\in H^1_0(\O)$, then $\|u_{\epsilon}\|_{H^1_0}\leq \|u\|_{H^1_0}$, and $u_{\epsilon}\rightarrow u$ in $H^1_0(\Omega)$, as $\e \rightarrow 0$.
\end{itemize}
\end{proposition}
\begin{proof}
The proof of these two statements are essentially the same. We only consider the second one.
Since $\O$ is bounded with boundary of class $C^2$, it is well known
that $-\Delta$,  with  the domain $ H^2(\O)\cap H^1_0(\O)$, is positive, self-adjoint, and it is the inverse of a compact operator.
Moreover, $-\Delta$ has an infinite sequence of positive eigenvalues $0<\l_1\leq \l_2\leq \cdots \leq \l_j\leq \cdots$,
and a corresponding sequence of eigenfunctions $\{e_j: j=1,2,\cdots \}$ that forms an
orthonormal basis for $L^2(\O)$, i.e., $-\Delta e_j=\l_j e_j$ with $e_j=0$ on $\partial \O$. Also, the sequence $\{ e_j: j=1,2,\cdots\}$ is an orthogonal basis for $H^1_0(\O)$. In addition, the standard
norm $\norm {u}_{H^1_0(\O)}$ is equivalent to $ \left (\sum_{j =
1}^\infty \l_j\abs {(u,e_j)}^2\right )^{1/2}$. Therefore, we consider
\begin{equation*} \norm {u}_{H^1_0(\O)}^2= \sum_{j = 1}^\infty
\l_j\abs {(u,e_j)}^2.
\end{equation*}
We now have
\begin{align*}
\|u_{\epsilon}\|^2_{H^1_0(\O)}=\sum_{j=1}^{\infty} \frac{\l_j}{(1+\epsilon \lambda_j)^2} |(u,e_j)|^2 \leq \sum_{j=1}^{\infty} \l_j |(u,e_j)|^2 = \|u \|^2_{H^1_0}.
\end{align*}
To see that $u_{\epsilon}\rightarrow u$ in $H^1_0(\Omega)$, we calculate
\begin{align*}
\|u_{\epsilon}-u\|^2_{H^1_0(\O)}&=\sum_{j=1}^{\infty}\Big(\frac{\epsilon \l_j}{1+\e \l_j} \Big)^2 \l_j |(u,e_j)|^2 \notag\\
&=\sum_{j=1}^{N}\Big(\frac{\epsilon \l_j}{1+\e \l_j} \Big)^2 \l_j |(u,e_j)|^2
+\sum_{j=N+1}^{\infty}\Big(\frac{\epsilon \l_j}{1+\e \l_j} \Big)^2 \l_j |(u,e_j)|^2 \notag\\
&\leq \epsilon^2 \sum_{j=1}^N \l_j^3 |(u,e_j)|^2  +\sum_{j=N+1}^{\infty}  \l_j |(u,e_j)|^2.
\end{align*}
Since $u\in H^1_0(\O)$, then $\sum_{j=1}^{\infty}  \l_j |(u,e_j)|^2 <\infty$. Thus,  by choosing $N$ large enough and then selecting  $\epsilon$ small enough, the conclusion follows.
\end{proof}

\smallskip

\begin{proposition} \label{norm-conveg}
Let $u\in L^p(\Omega)$ with $1<p<\infty$, then $\|u_{\epsilon}\|_p \leq \|u\|_p$ and  $u_{\epsilon}\rightarrow u$ in $L^p(\O)$, as $\e \rightarrow 0$.
\end{proposition}

\begin{proof}
By the definition of $u_{\epsilon}$, we have
\begin{align*}
\begin{cases}
u_{\epsilon}-\epsilon \Delta u_{\epsilon}=u \text{\;\;in\;\;} \Omega; \\
u_{\epsilon}=0 \text{\;\;on\;\;} \partial \Omega.
\end{cases}
\end{align*}
From the standard elliptic theory (see \cite{Gilbarg-Trudinger}), $u_{\e}\in W^{2,p}(\O)$ whenever $u\in L^p(\O)$, for $1<p<\infty$.

\emph{Case 1: $2\leq p <\infty$}.

Multiply the equation by $|u_{\epsilon}|^{p-2}u_{\epsilon}$ and integration by parts:
\begin{align*}
\|u_{\epsilon}\|_p^p+\epsilon \int_{\O} \nabla u_{\epsilon} \cdot \nabla (|u_{\epsilon}|^{p-2}u_{\epsilon})dx=\int_{\O} u |u_{\epsilon}|^{p-2}u_{\epsilon} dx.
\end{align*}
A straightforward calculation gives
\begin{align*}
\|u_{\epsilon}\|_p^p+\epsilon (p-1) \int_{\O} |u_{\epsilon}|^{p-2} |\nabla u_{\epsilon}|^2  dx=\int_{\O} u |u_{\epsilon}|^{p-2}u_{\epsilon} dx.
\end{align*}
Since the middle term is positive, one has
\begin{align*}
\|u_{\epsilon}\|_p^p \leq \int_{\O} |u| |u_{\epsilon}|^{p-1} dx \leq \|u\|_p \|u_{\epsilon}\|_p^{p-1}.
\end{align*}
It follows that
\begin{align} \label{moli-bound}
\|u_{\epsilon}\|_p \leq \|u\|_p, \text{\;\;for\;\;} 2\leq p<\infty.
\end{align}

To show $u_{\e}\rightarrow u$ in $L^p$, we argue by contradiction. Suppose not, then there exist $\delta>0$ and a subsequence $\{u_{\e_j}\}$ such that
\begin{align} \label{vilat}
\|u_{\e_j}-u\|_p\geq \delta, \text{\;\;for all\;\;} j=1,2,\cdots.
\end{align}
Due to the uniform bound (\ref{moli-bound}) and the fact $u_\epsilon \rightarrow u$ strongly in $L^2(\Omega)$ from Proposition \ref{H1conv}, we can extract a further subsequence $\{u_{\e_{j_k}}\}$ such that $u_{\e_{j_k}} \rightarrow u$ weakly in $L^p(\Omega)$. As a result, the weak convergence and the uniform bound (\ref{moli-bound}) imply
\begin{align*}
\|u\|_p\leq \liminf_{k \rightarrow 0} \|u_{\e_{j_k}}\|_p
\leq \limsup_{k \rightarrow 0} \|u_{\e_{j_k}}\|_p \leq \|u\|_p.
\end{align*}
This shows
\begin{align*}
\lim_{k \rightarrow 0} \|u_{\e_{j_k}}\|_p=\|u\|_p.
\end{align*}
Since we already know $u_{\e_{j_k}} \rightarrow u$ weakly in $L^p(\Omega)$, it follows that $u_{\e_{j_k}} \rightarrow u$ strongly in $L^p(\Omega)$, which violates (\ref{vilat}).

\smallskip

\emph{Case 2: $1< p <2$}.

In this case, the conjugate $2<p^*<\infty$. We calculate
\begin{align*}
\|u_{\epsilon}\|_p&=\sup_{\| \varphi \|_{p^*}=1}\int_{\Omega}u_{\epsilon}\varphi dx
=\sup_{\| \varphi \|_{p^*}=1}\int_{\Omega}u_{\epsilon}((I-\epsilon \Delta) \varphi_{\epsilon}) dx
=\sup_{\| \varphi \|_{p^*}=1}\int_{\Omega} ((I-\epsilon \Delta) u_{\epsilon}) \varphi_{\epsilon} dx  \notag\\
&=\sup_{\| \varphi \|_{p^*}=1}\int_{\Omega} u \varphi_{\epsilon} dx \leq \sup_{\| \varphi \|_{p^*}=1}
\|u\|_p \|\varphi_{\epsilon}\|_{p^*} \leq  \sup_{\| \varphi \|_{p^*}=1}
\|u\|_p \|\varphi\|_{p^*}=\|u\|_p,
\end{align*}
where the uniform bound (\ref{moli-bound}) has been used.

Finally, arguing as in Case 1, the convergence of $u_{\epsilon}$ in $L^p(\Omega)$ for $1<p<2$ follows.
\end{proof}

\smallskip

\noindent {\bf Acknowledgment.} The research of Y. Guo and E. S. Titi was supported in part by the NSF grants DMS--1009950, DMS--1109640, DMS--1109645, and also  by the Minerva Stiftung/Foundation. The research of  S. Sakuntasathien was supported by Faculty of Science, Silpakorn University. The research of  D. Toundykov was supported in part by the NSF grant DMS-1211232.

The authors are grateful to Jinkai Li for helpful discussions, and to the anonymous referee for many insightful comments and for suggesting substantial improvements to the original proofs.

\bibliographystyle{amsplain}

\end{document}